\documentclass[15pt]{article}
\usepackage{fancyhdr}
\usepackage{fancyvrb}
\usepackage{tikz}
\usepackage{pgfplots}
\usepackage{amsmath}
\usepackage{flexisym}
\usepackage{breqn}
\usepackage{amssymb}
\usepackage{amsthm}
\usepackage{bbm}
\usepackage{url}
\usepackage{enumerate}
\usepackage{hyperref}
\usepackage[style = numeric, doi=false,isbn=false,url = false, maxbibnames = 99]{biblatex}
\renewbibmacro{in:}{}
\pgfplotsset{compat=1.11}
\usepackage{graphicx}
\newtheorem{Theorem}{Theorem}
\newtheorem{definition}[Theorem]{Definition}
\newtheorem{lemma}[Theorem]{Lemma}
\newtheorem{remark}[Theorem]{Remark}
\newtheorem{proposition}[Theorem]{Proposition}
\providecommand{\keywords}[1]
{
  \small	
  \textbf{\textit{Keywords---}} #1
}
\numberwithin{equation}{section}
\numberwithin{Theorem}{section}
\DeclareMathOperator*{\esssup}{ess\,sup}

\begin{document}
\title{On the Functional L\'{e}vy-It\^{o} Stochastic Calculus}
\author{Christian Houdr\'{e}\footnote{School of Mathematics, Georgia Institute of Technology, Atlanta, GA, 30332-0160, USA, houdre@math.gatech.edu. Research supported in part by the Simons Foundation grant \#524678.}\,\, and Jorge V\'{i}quez\footnote{School of Mathematics, Georgia Institute of Technology, Atlanta, Georgia, 30332-0160, USA, jviquez6@gatech.edu. Partially funded by the Department of Mathematics of the University of Costa Rica.}}

\maketitle

\begin{abstract}
Several versions of It\^{o}'s formula have been obtained in the context of the functional stochastic calculus. Here, we revisit this topic in two ways. First, by defining a notion of derivative along a functional, we extend the setting of the (semimartingale) functional It\^{o}'s formula and corresponding calculus. Second, for L\'{e}vy processes, an optimal local-time based It\^{o}'s formula is obtained. Some quick applications are then given.
\end{abstract}

\\


\section{Introduction}

Dupire \cite{Dupire09} defined notions of vertical and horizontal derivatives allowing for a functional version of It\^{o}'s formula useful in applications. These definitions expanded upon previous ones such as those of Ahn \cite{Ahn}, which involved Fr\'{e}chet derivatives and required considerations of changes along the whole trajectory of a process. When the corresponding derivatives exist, in both of these approaches, Ji and Yang \cite{Comparison} showed that they are the same, although in general, the existence of Fr\'{e}chet derivatives is a stronger requirement. Cont and Fourni\'{e} \cite{Barcelona},\cite{Cont13},\cite{CONT20101043},\cite{Fournie10} revisited these definitions in light of the pathwise framework pioneered by F\"{o}llmer \cite{Follmer81}, and extended the formula to general c\`{a}dl\`{a}g functions with bounded quadratic variation along a sequence of partitions. On the other hand, Saporito \cite{Saporito} extended the functional It\^{o} formula to obtain a Meyer-Tanaka theorem under regularity conditions in the functional derivatives. Moreover, Levental, Schroder, and Sinha \cite{LEVENTAL13} obtained a version of the It\^{o} formula for functionals of general semimartingales, later applied by Siu \cite{ConvexRisk} to study convex risk measures. Other versions specialized to jump diffusion processes, and making use of Fr\'{e}chet derivatives in an $L^2$ space, have been developed by Ba\~{n}os, Cordoni, Di Nunno, Di Persio, and Røse in \cite{MemoryJumps}.

Further extensions and alternative approaches have also been proposed. For example, Oberhauser \cite{Oberhauser} describes general conditions that a functional and its derivatives have to satisfy in order for a functional It\^{o} formula to hold true, and Litterer and Oberhauser \cite{OberhauserStratonovich} develop an iterated integral extension for differentiable functionals in the Stratonovich setting. Another approach is that of Cosso and Russo \cite{CossoRusso} which uses calculus via regularization to obtain an It\^{o} formula for functionals of processes with continuous paths, without requiring to extend the domain of the functional to processes with c\`{a}dl\`{a}g paths. Still, using calculus via regularization, Bouchard, Loeper, and Tan \cite{Bouchard} extended the functional It\^{o} formula to continuous weak Dirichlet processes, under less stringent regularity conditions. Let us also mention the work of Buckdahn, Ma, and Zhang \cite{PathwiseTaylor} who obtained a Taylor expansion for path functionals using derivatives defined from the semimartingale decomposition of It\^{o} processes. Moreover, in the context of path dependent partial differential equations (PPDE), Ekren, Touzi, and Zhang\cite{ViscosityI}, \cite{ViscosityII} defined functional derivatives using the functionals that allow an It\^{o} formula to be established, and used them to define viscosity solutions for PPDEs. Keller \cite{ViscosityKeller} extended these derivatives to the discontinuous setting to allow for path dependent integro-differential equations. Then, similar derivatives were defined by Keller and Zhang \cite{RoughPath} in order to address viscosity solutions in the context of rough paths.\\

Let us briefly describe the content of our notes. At first, Section 2 provides an introduction to some aspects of the functional It\^{o} calculus. The functionals to be used and the space they act upon are defined, together with the corresponding functional derivative. A functional Fisk-Stratonovich formula is also obtained. Section 3 introduces a notion of derivative in the direction of a functional, and relates it to the horizontal derivative. Under smoothness conditions it is then expressed via the horizontal and vertical derivatives. In particular, when the functional is h-Lipshitz, this derivative is well defined. Next, Section 4, derives an It\^{o} formula when the underlying path is given by a L\'{e}vy process, extending the optimal, local-time based, It\^{o}'s formula of Eisenbaum and Walsh \cite{Eisenbaum09} to the functional setting. Section 5 discusses some simple applications.\\

\section{Notations and Definitions}

This section presents some of the definitions and concepts that will be dealt with throughout the rest of these notes, many originate in \cite{Dupire09}, some from \cite{Bouchard}, and some are original. We work with the space $D([0,T],\mathbb{R}^d)$ of c\`{a}dl\`{a}g functions $w$ with domain $[0,T]$ and codomain $\mathbb{R}^d$. Then, $(D([0,T],\mathbb{R}^d),\mathcal{F},(\mathcal{F}_t)_{t \in [0,T]},\mathbb{P})$ is our underlying probability space, satisfying the usual conditions, in which the stochastic process $(X(t))_{t \in [0,T]}$ with $X(t,w) = w(t)$ is adapted. In the manuscript, different probability measures $\mathbf{P}$, over the same filtration and for which the process $(X(t))_{t \in [0,T]}$ is a semimartingale will also be used. Throughout, let $w_{\wedge t} \in D([0,T],\mathbb{R}^d)$ be defined via \[w_{\wedge t}(s) := w(s)\mathbbm{1}_{[0,t)}(s)+w(t)\mathbbm{1}_{[t,T]}(s),\] 
\noindent and let $w_{\wedge t}^h \in D([0,T], \mathbb{R}^d)$ be defined via 
\[w_{\wedge t}^h(s) := w_{\wedge t}(s)+h\mathbbm{1}_{[t,T]}(s).\]

Below, the main objects of study are functionals $F:[0,T]\times D([0,T],\mathbb{R}^d) \to \mathbb{R}$ which are \emph{non-anticipative} in that,

\begin{dmath*}
    F(t,x) = F(t,x_{\wedge t})
\end{dmath*},
\noindent
and which are measurable with respect to the product $\sigma$-field $\mathcal{B}([0,T])\otimes \mathcal{F}$, where $\mathcal{B}([0,T])$ is the Borel $\sigma$-field of $[0,T]$. Moreover, $[0,T]\times D([0,T],\mathbb{R}^d)$ is equipped with a pseudometric $d_*$ defined via $d_*((t,w),(s,v)) := ||(t,w)-(s,v)||_* := |t-s|+d_D(w_{\wedge t},v_{\wedge s})$, where $d_D$ could be any metric in $D([0,T],\mathbb{R}^d)$. Below, the metric will be given by ``the infinity norm", $\lvert\lvert\cdot\rvert\rvert_\infty$, but it could also be, for example, given by a norm associated with the Skorokhod topology.\\

A functional $F$ is said to be \emph{(right)-horizontally differentiable} at $(t,w) \in [0,T]\times D([0,T], \mathbb{R}^d)$, $t < T$, if the following limit exists:

\begin{align}
    DF(t,w) &:= \lim_{h\to 0^+} \frac{F(t+h,w_{\wedge t})-F(t,w_{\wedge t})}{h}.
\end{align}
Thus, if $F$ is horizontally differentiable for any pair $(t,w)$ the functional $DF: [0,T)\times D([0,T],\mathbb{R}^d)$ associating to each pair its horizontal derivative is well defined.\\

Similarly, a functional is said to be \emph{space differentiable} at $(t,w) \in [0,T]\times D([0,T],\mathbb{R}^d)$, in the direction of a canonical vector $e_i \in \mathbb{R}^d$, $i = 1,...,d$, if the following limit exists:

\begin{align}
    \partial_i F(t,w) &:= \lim_{h \to 0} \frac{F(t,w_{\wedge t}^{he_i})-F(t,w_{\wedge t})}{h}.
\end{align}

Again, if this limit exists for any pair $(t,w)$, and any $e_i$, one defines a functional $\partial_i F: [0,T]\times D([0,T],\mathbb{R}^d)$ which associates to each pair its derivative with respect to the $i$-th canonical vector. With the help of these functionals one then defines the gradient of the functional as:

\begin{align}
    \nabla F(t,w) := (\partial_1 F(t,w),\partial_2 F(t,w),...,\partial_d F(t,w)).
\end{align}
\\
Next, let us recall the regularity conditions imposed on $F$ in \cite{Cont13} in order to obtain the corresponding It\^{o}'s formula. A functional $F$ is said to be \emph{boundedness-preserving} if for every compact set $K \subset \mathbb{R}^d$, and any $t^* \in [0,T]$, there exists a constant $C_{K,t^*}$ such that for any function with co-domain $K$, $|F(t,w_{\wedge t})| \leq C_{K,t^*}$, for all $t \leq t_*$. A functional is said to be \emph{fixed-time continuous at $(t,w)$}  if $F(t,\cdot)$ is continuous at $w_{\wedge t}$ as a function of its second variable i.e., if for all $\epsilon> 0$, there exists $\delta > 0$, such that $\vert\vert v_{\wedge t}-w_{\wedge t}\vert\vert_\infty < \delta$ implies $\vert F(t,w_{\wedge t})-F(t,v_{\wedge t})\vert < \epsilon$. In a similar way, it will be called \emph{fixed-time continuous} if it is fixed time continuous for every pair $(t,w) \in [0,T]\times D([0,T], \mathbb{R}^d)$. Finally, \emph{left-continuity (in time) at $(t,w)$} is used for functionals $F$ such that for all $\epsilon > 0$, there exists $\delta > 0$ such that $s<t$ with $||(t,w)-(s,v)||_*<\delta$ implies $\vert F(t,w_{\wedge t})-F(s,v_{\wedge s})\vert <\epsilon$. Again, $F$ will then be called \emph{left-continuous (in time)} if it is left-continuous for every pair $(t,w) \in [0,T]\times D([0,T],\mathbb{R}^d)$.\\

Let $C^{j,k} := C^{j,k}([0,T]\times D([0,T],\mathbb{R}^d))$, $j,k \in \{1,2,...\}$, be the set of left-continuous (in time) functionals which are $j$-times horizontally differentiable, $k$-times space differentiable, with the horizontal derivatives continuous at fixed times, while the space derivatives are left-continuous in time, and all these functionals are boundedness-preserving. Additionally $C^{0,k}$ and $C^{j,0}$ denote the set of boundedness-preserving, left-continuous functionals that satisfy the differentiability requirements given by the non-zero super index. In the same way, $C^{0,0}$ corresponds to the set of boundedness-preserving, left-continuous (in time) functionals.\\

Within this framework, the following result was proved for continuous functionals in \cite{Dupire09}, and for left-continuous ones in \cite{Cont13}. Below, $([X](t))_{t \in [0,T]}$ denotes the quadratic covariation matrix of the process $X = (X(t))_{t \in [0,T]}$.

\begin{Theorem}[Functional It\^{o} Formula]
\label{Functional Ito Formula}
Let $F \in C^{1,2}$, and let $\mathbf{P}$ be a probability measure such that $(X(t))_{t \in [0,T]}$ is $\mathbf{P}$-a.s. a continuous semimartingale. Then $\mathbf{P}-$a.s.:

\begin{multline}
    F(T,X_{\wedge T})-F(0,X_0) = \int_0^T DF(t,X_{\wedge t})\,\mathrm{d}t+\int_0^T \nabla F(t,X_{\wedge t})\cdot\mathrm{d}X(t)+\frac{1}{2}\int_0^T Tr(\nabla^2 F(t,X_{\wedge t})d[X](t)).
\end{multline}

\end{Theorem}

Moreover, \cite{Fournie10}, \cite{Barcelona}, and \cite{LEVENTAL13}, extended $(2.4)$ to measures $\mathbf{P}$ for which $(X(t))_{t \in [0,T]}$ is a c\`{a}dl\`{a}g semimartingale, in which case $\mathbf{P}$-a.s.:

\begin{multline}
     F(T,X_{\wedge T})-F(0,X_0) = \int_0^T DF(t,X_{\wedge t^-})\,\mathrm{d}t+\int_0^T \nabla F(t,X_{\wedge t^-}) \cdot \mathrm{d}X(t)+\frac{1}{2}\int_0^T Tr(\nabla^2 F(t,X_{\wedge t^-})d[X]^c(t))\\
    +\sum_{t \in [0,T]} F(t,X_{\wedge t})-F(t,X_{\wedge t^-})-\langle \nabla F(t,w_{\wedge t^-}),\Delta X(t)\rangle,
\end{multline}

\noindent where $\Delta X(t) = X(t)-X(t^-)$, while  $[X]^c(t)$ is the continuous part of the quadratic covariation matrix, i.e., $([X]^c(t))_{i,j} = [X^i,X^j](t)-\sum_{s\in [0,t]} \Delta X^i(t)\Delta X^j(t)$, $i,j = 1,...,d$.\\

Besides the pathwise derivatives studied in  \cite{Cont13}, additional results have appeared in the literature. For example, \cite{Oberhauser} showed that both the quadratic variation and the stochastic It\^{o} integral are given by differentiable functionals over the set of continuous functions $C([0,T], \mathbb{R}^d)$. Earlier, Ekren, Keller, Touzi, and Zhang \cite{Zhang} similarly defined the functional derivatives $\partial_t F, \partial_w F, \partial_{ww} F$ as the continuous bounded functionals satisfying the relation:

\begin{multline}
    F(T,X_{\wedge T})-F(0,X_0) = \int_0^T \partial_tF(t,X_{\wedge t})\,\mathrm{d}t+\int_0^T \partial_w F(t,X_{\wedge t})\cdot \mathrm{d}X(t)\\
    + \frac{1}{2}\int_0^T Tr(\partial_{ww}F(t,X_{\wedge t})\,\mathrm{d}[X](t)) \condition{$\mathbf{P}$ a.s. for $\mathbf{P} \in M(X)$},
\end{multline}

\noindent where $M(X)$ is the family of probability measures $\mathbf{P}$ under which $X$ is a continuous semimartingale with bounded drift and diffusion. Observe that if $\mathbf{P}$-almost surely, $X$ is right-differentiable on $[s,T)$ with right derivative $y$, then a.s. $dF(t,X_{\wedge t}) = \partial_t F(t,X_{\wedge t})+\langle\partial_w F(t,X_{\wedge t}),y(t)\rangle$, which implies that $\partial_w F(t,X_{\wedge t})$ is the co-invariant derivative of $F$ in the sense of Kim \cite{Kim}. Keller \cite{ViscosityKeller} further extended this definition to include measures for which $X$ is a c\`{a}dl\`{a}g semimartingale with bounded jumps.

 Cosso and Russo \cite{CossoRusso} use calculus by regularization to obtain an It\^{o} formula in a different framework, with functionals in the space $\mathcal{C}([0,T],\mathbb{R}^d)$ of bounded functions continuous on $[0,T)$ with a possible discontinuity at $T$. This space is endowed with the topology of the uniform convergence on compact sets, and this framework allows to part with the requirement of extending functionals to the entire space of c\`{a}dl\`{a}g functions.\\
 
A common alternative to the It\^{o} integral is given by the Fisk-Stratonovich integral. Given an adapted c\`{a}dl\`{a}g process $X$, and a semimartingale $Y$, both taking values in $\mathbb{R}$, the Fisk-Stratonovich integral is given by:

\begin{align}
    \int_0^t X(s^-)\circ \mathrm{d}Y(s) = \int_0^t X(s^-)\cdot\mathrm{d}Y(s)+\frac{1}{2}[X,Y]^c(t),
\end{align}

\noindent where $[X,Y]^c(t)$ is the continuous part of the quadratic covariation of $X(t)$ and $Y(t)$.

The previous definition can then be extended to multivariate processes. Given an adapted c\`{a}dl\`{a}g process $X = (X^1,...,X^d)$, and a semimartingale $Y = (Y^1,...,Y^d)$, the Fisk-Stratonovich integral is then given by:

\begin{align*}
    \int_0^t X(s^-)\circ \mathrm{d}Y(s) &= \sum_{i = 1}^d \int_0^t X^i(s^-)\circ \mathrm{d}Y^i(s).
\end{align*}

Provided the equality between $\int_0^t \nabla^2 Tr(F(t,X_{\wedge t})\,\mathrm{d}[X]_t)$ and $\sum_{i = 1}^d\sum_{j = 1}^d [\partial_i F(\cdot,X),X^j]^c(t)$ is established, the functional It\^{o} formula rewrites as:

\begin{multline}
F(T,X_{\wedge T})-F(0,X_0) = \int_0^T DF(t,X_{\wedge t})\,\mathrm{d}t+ \int_0^T \nabla F(t,X_{\wedge t^-})\circ \mathrm{d}X(t)\\
+\sum_{t \in [0,T]} \left(F(t,X_{\wedge t})-F(t,X_{\wedge t^-})-\langle \nabla F(t,w_{\wedge t^-}),\Delta X(t)\rangle\right).
\end{multline}

The above allows to write the Fisk-Stratonovich integral in a manner analogous to the classical setting. However, the expression for $[\partial_i F(\cdot,X),X^i]$ is not immediate, and conditions under which it can be re-written as the aforementioned integral have to be established. This can be done when $\nabla F \in C^{1,2}$, by reapplying Theorem \ref{Functional Ito Formula} to $\nabla F$, as in \cite{OberhauserStratonovich}. Moreover, if $\partial_i F(t,x_{\wedge t}) = \int_0^t g_i(s,x_{\wedge s})\mathrm{d}I(s)$, with $I:[0,T]\to \mathbb{R}^d$ a bounded variation process, and $g$ a boundedness-preserving functional such that $\partial_{ij} F$ is continuous, and $\partial_{ij}F(t,x_{\wedge t}) = \int_0^t \partial_j g_i(s,x_{\wedge s})\mathrm{d}I(s)$, then it follows directly from \cite[Theorem V.19]{Protter} that $(2.8)$ holds true as well. In particular, such conditions are satisfied when $h\in C^{1,0}$, with $h(t,w_{\wedge t},x) = \nabla F(t,w_{\wedge t}^{x-w(t)})$, and this requirement has been previously used by Saporito \cite{Saporito} to obtain a local time It\^{o} formula for functionals of continuous paths. Applying the next theorem to $\partial_i F \in C^{1,1}$, for all $i \in \{1,...,d\}$, allows to write $[\nabla F(\cdot,X),X]$ in a manner that recovers $(2.8)$ from the functional It\^{o} formula $(2.4)$.\\

To start with, a technical lemma is needed:

\begin{lemma}
Let $F \in C^{0,0}$, then for all $x \in D([0,T], \mathbb{R}^d)$, and all $\epsilon >0$, there exists $\delta > 0$ such that if $s \leq t$ and if $y \in D([0,T],\mathbb{R}^d)$ then:

\begin{dmath}
\vert\vert (s,y_{\wedge s})-(t,x_{\wedge t^-})\vert\vert_* < \delta \implies \vert F(s,y_{\wedge s})-F(t,x_{\wedge t^-})\vert < \epsilon
\end{dmath}
\end{lemma}

\begin{proof}
The proof follows the method used to prove the $\Lambda$-lemma in \cite{Dupire09}. Fix $x \in D([0,T],\mathbb{R}^d)$, and for the purpose of contradiction, assume that there exist $\epsilon > 0$, and sequences $(s_n)_{n \geq 1}, (t_n)_{n \geq 1}, s_n \leq t_n$ both contained in $[0,T]$, such that

\begin{align}
\begin{split}
    \forall\, n \geq 1, \exists y_n \in D([0,T],\mathbb{R}^d):&\vert\vert (s_n,(y_n)_{\wedge s_n})-(t_n,x_{\wedge t_n^-})\vert\vert_* < 1/n,\, \vert F(s_n, (y_n)_{\wedge s_n})-F(t_n,x_{\wedge t_n^-})\vert \geq \epsilon\\
\end{split}
\end{align}
Since $[0,T]$ is compact, there exist $t^* \in [0,T]$ and an increasing subsequence $(t_{n_k})$ of $(t_n)$ such that $t_{n_k} \to t^*$, then

\begin{align}
\begin{split}
    \vert F(s_{n_k},(y_{n_k})_{\wedge s_{n_k}})-F(t_{n_k},x_{\wedge t_{n_k}^-})\vert \leq  \vert F(s_{n_k},(y_{n_k})_{\wedge s_{n_k}})-F(t^*,x_{\wedge t^{*-}})\vert+\vert F(t_{n_k},x_{\wedge t_{n_k}^-})-F(t^*,x_{\wedge t^{*-}})\vert\\
\end{split}
\end{align}

Since $t_{n_k} \to t^*$, then $s_{n_k} \to t^*$ as well, and $(t_{n_k},x_{\wedge t_{n_k}^-}) \to (t^*,x_{\wedge t^{*-}})$  in the pseudometric $d_*$. Therefore, the last term in $(2.11)$ converges to $0$. Similarly,

\begin{align*}
    \vert\vert (s_{n_k},(y_{n_k})_{\wedge s_{n_k}})-(t^*,x_{\wedge t^{*-}})\vert\vert_* \leq \vert\vert (s_{n_k},(y_{n_k})_{\wedge s_{n_k}})-(t_{n_k},x_{\wedge t_{n_k}}^-)\vert\vert_*+\vert\vert (t_{n_k},x_{\wedge t_{n_k}}^-)-(t^*,x_{\wedge t^{*-}})\vert\vert_*\to 0.
\end{align*}

\noindent Since $F$ is left continuous at $(t^*,x_{\wedge t^*})$ this contradicts $(2.10)$ proving the claim.
\end{proof}

\begin{Theorem}[Fisk-Stratonovich Formula]
\label{Frisk-Stratonovich}

Let $F \in C^{1,1}$, then $\mathbf{P}$-a.s.:

\begin{dmath}
[F(\cdot,X),X^j](t) = \sum_{i = 1}^d \int_0^t \partial_i F(s,X_{\wedge s^-})\mathrm{d}[X^i,X^j]^c(s)+\sum_{s \in [0,t]} \Delta F(t,X_{\wedge t})\Delta X^j(s)
\end{dmath}.
\end{Theorem}

\begin{proof}

Since $F$ is left-continuous, horizontally differentiable, with a locally bounded horizontal derivative one has $F(t+h,X_{\wedge t})-F(t,X_{\wedge t}) = \int_t^{t+h} D F(s,X_{\wedge s})\,\mathrm{d}s$. Thus if one redefines $F(t,X_{\wedge t})$ by substracting $\int_0^t D F(s,X_{\wedge s})\,\mathrm{d}s$ (which we do), and since this last integral is of bounded variation, the resulting functional (which we still denote by $F$) will have the same covariation with $X$, and will be constant along constant paths, i.e, $F(t+s,w_{\wedge t}) = F(t,w_{\wedge t})$, for any pair $(t,w)$, with $t < T$.\\

To start, assume that a.s.\ there exists $M > 0$ such that, $\vert\vert X(t)\vert\vert_2 \leq M$, for all $t \in [0,T]$, with $\vert\vert\cdot\vert\vert_2$ the usual Euclidean norm in $\mathbb{R}^d$. Next, as in \cite[Lemma A.3]{Cont13} take $\{\tau_n\}_{n \geq 1}$ a nested sequence of partitions $\tau_n = \{t^n_0,...,t^n_{k_n}\}$ such that the c\`{a}dl\`{a}g approximations to $X$: $X^n(t) = \sum_{i = 0}^{k_n-1} X(t^{n-}_{i+1})\mathbbm{1}_{[t^n_i,t^n_{i+1})}(t)+X(T)\mathbbm{1}_{\{T\}}(t)$ converge uniformly to the original function, the oscillation of the original $X$ inside each subinterval $[t^n_i,t^n_{i+1})$ converges uniformly to $0$, and the jumps at times not contained in $\tau_n$ converge uniformly to 0. \\

From Lemma 2.2. it follows that $F_n(t) = F(t,X^n_{\wedge t})$ converges uniformly almost everywhere to $F(t,X_{\wedge t})$. Take $k < l$, then, for $n$ large enough, given the construction in \cite{Cont13}, the intervals from the partition $\tau_n$ will always be nested on the ones from $\tau_l$, therefore:

\begin{align*}
    \sum_{i = 0}^{k_n-1} &[(F_l(t^n_{i+1})-F_k(t^n_{i+1}))-(F_l(t^n_i)-F_k(t^n_i))]^2\\
    =\,&\sum_{i = 0}^{k_n-1} [(F_l(t^n_{i+1})-F_l(t^n_i))-(F_k(t^n_{i+1})-F_k(t^n_i))]^2\\
    =\,&\sum_{i = 0}^{k_n-1} [(F_l(t^n_{i+1})-F_l(t^n_i))-(F_k(t^n_{i+1})-F_k(t^n_i))]^2\mathbbm{1}_{\tau_l}(t^n_{i+1})\\
    \leq\,& 2\sum_{i = 0}^{k_n-1} (F_l(t^n_{i+1})-F_k(t^n_{i+1}))^2+(F_l(t^n_i)-F_k(t^n_i))^2\mathbbm{1}_{\tau_l}(t^n_{i+1}),
\end{align*}

Since this last sum is finite and since $(X^n(t^-))_{n \geq 1}$ forms a Cauchy sequence, the original expression converges to $0$. From this, by proceeding as in the proof of \cite[Theorem V.19]{Protter}, it follows that the measures $(d[F_k])_{k \geq 1}$ form a Cauchy sequence with respect to the total variation distance. Since $F_n$ is a step function with a finite number of discontinuities, one has that $[F_n]_t < \infty$, and moreover:

\begin{align*}
[F_n](t)-(F(T,X_{\wedge T})-F(T,X_{\wedge T^-}))^2\mathbbm{1}_{\{T\}}(t) &= \sum_{i = 0}^{k_n-1} (F(t^n_{i+1} \wedge t, X^n_{\wedge (t^n_{i+1})})-F(t^n_i\wedge t,X^n_{\wedge (t^n_i)}))^2\\
&= \sum_{i = 0}^{k_n-1} \left(F\Big(t^n_{i+1}\wedge t ,(X^n_{\wedge t^{n-}_{i+1}})^{X(t^{n-}_{i+2})-X(t^{n-}_{i+1})}\right)-F(t^n_{i+1}\wedge t,X^n_{\wedge t^{n-}_{i+1}})\Big)^2.
\end{align*}

As is standard, e.g., see the proof of \cite[Theorem V.18]{Protter}, Theorem V.18], for $\epsilon > 0$ split $[0,T]$ into two sets $A,B \subseteq [0,T]$ such that a.s.\@ $A$ is finite and $\sum_{t \in B} \vert\vert\Delta X(t)\vert\vert_2^2 \leq \epsilon^2/2$. Then, since $A$ is finite:

\begin{multline*}
\sum_{\{i:A\cap [t^n_{i+1},t^n_{i+2}) \neq \emptyset\}} \left(F\left(t^n_{i+1}\wedge t,(X^n)_{\wedge t^{n-}_{i+1}}^{X(t^{n-}_{i+2})-X(t^{n-}_{i+1})}\right)-F(t^n_{i+1}\wedge t,X^n_{\wedge t^{n-}_{i+1}})\right)^2 \\
\underset{n\to \infty}{\longrightarrow} \sum_{ s \in [0,t] \cap A} \left(F(s\wedge t,X_{\wedge s})-F(s\wedge t, X_{\wedge s^-})\right)^2.
\end{multline*}

\noindent Moreover:

\begin{dmath}
\sum_{\{i:[t^n_{i+1},t^n_{i+2}) \cap A = \emptyset\}} \left(F\left(t^n_{i+1}\wedge t,(X^n)_{\wedge t^{n-}_{i+1}}^{X(t^{n-}_{i+2})-X(t^{n-}_{i+1})}\right)-F(t^n_{i+1}\wedge t, X^n_{\wedge t^{n-}_{i+1}})\right)^2
=\sum_{\{i:[t^n_{i+1},t^n_{i+2}) \cap A = \emptyset\}} \left(\langle\nabla F(t_{i+1},X^n_{\wedge t^{n-}_{i+1}}), X(t^{n-}_{i+2})-X(t^{n-}_{i+1})\rangle+R(t^n_{i+1},X^n_{\wedge t^{n-}_{i+1}},X(t^{n-}_{i+2}))\right)^2
\end{dmath},
\noindent where $R$ is the remainder when applying a first order Taylor expansion to the first line of $(2.13)$.

Define the matrix $(D^n(t))_{i,j} = \nabla F(t,X^n_{\wedge t^{n^-}})\nabla F(t,X^n_{\wedge t^{n^-}})^T$, i.e. the matrix with entries $\partial_i F(t,X^n_{\wedge t^{n-}})\partial_j F(t,X^n_{\wedge t^{n-}})$, and observe that its entries are bounded left-continuous functions, thus the last term in $(2.13)$ becomes:

\begin{align}
\sum_{\{i:[t^n_{i+1},t^n_{i+2}) \cap A = \emptyset\}} Tr(D^n(t^n_{i+1})(X(t^{n-}_{i+2})&-X(t^{n-}_{i+1}))(X(t^{n-}_{i+2})-X(t^{n-}_{i+1}))^T)+S_n(t^n_{i+1},X(t^{n-}_{i+2})),
\end{align}

\noindent where 

\begin{align}
S_n(t^n_{i+1},X(t^{n-}_{i+2})) := r_n(t^n_{i+1}, X(t^{n-}_{i+2})) \vert\vert X(t^{n-}_{i+2})-X(^n t_{i+1}^-)\vert\vert_2,
\end{align}

\noindent with
\begin{align}
r_n(t,x) &:= \frac{(f_n(t,x)-f_n(t,X(t^-)))^2-Tr(D^n(t)(x-X(t^-))(x-X(t^-))^T)}{\vert\vert x-X(t^-)\vert\vert_2^2},
\end{align}
and where $f_n(t,x) := F\left(t,(X^n_{\wedge t^-})^{x-X(t^-)}\right)$. First, by the previous Taylor expansion,

\noindent$\lim_{||x-X(t^-)||_2\to 0} r_n(t,x)$ exists. Next, let $M_D > 0$ be such that $\vert (D^n(t))_{i,j}\vert \leq M_D$, for all $i,j \in \{1,2,...,d\}$, $n\geq 1$, $t\in [0,T]$. Since the space derivative is boundedness-preserving, for $x \leq M$ the following uniform bound holds true:

\begin{align*}
    \vert r_n(t,x)\vert &\leq \frac{c^2\vert\vert x-X(t^-)\vert\vert^2+M_D \vert\vert x-X(t^-)\vert\vert_2^2}{\vert\vert x-X(t^-)\vert\vert_2^2}= c^2+M_D = C.
\end{align*}

If $X^-$ is the process given by $X^-(t) = X(t^-)$ for all $t \in [0,T]$, then $[X^-](t) = [X](t)-(\Delta X(T))^2\mathbbm{1}_{\{t = T\}}(t)$, thus $(2.13)$ is equal to:

\begin{align}
\sum_{i = 0}^{k_n-1} &Tr(D^n(t^n_{i+1}) (X(t^{n-}_{i+2})-X(t^{n-}_{i+1}))(X(t^{n-}_{i+2})-X(t^{n-}_{i+1}))^T) \nonumber\\
&-\sum_{\{i:[t^n_{i+1},t^n_{i+2}) \cap A \neq \emptyset\}} Tr(D^n(t^n_{i+1}) (X(t^{n-}_{i+2})-X(t^{n-}_{i+1}))(X(t^{n-}_{i+2})-X(t^{n-}_{i+1}))^T)\nonumber\\
&+\sum_{\{i:[t^n_{i+1},t^n_{i+2}) \cap A = \emptyset\}}S_n(t^n_{i+1},X(t^{n-}_{i+2})).
\end{align}

The first sum above, converges to $\int_{[0,t)} Tr(\nabla F(s,X_{\wedge s^-})\nabla F(s,\wedge s^-)^T \mathrm{d}[X]_s)$, since the derivatives are bounded left-continuous. For the error term in $(2.17)$, observe that:

\begin{align*}
\limsup_n&\sum_{\{i:[t^n_{i+1},t^n_{i+2}) \cap A = \emptyset\}}S_n(t^n_{i+1},X(t^{n-}_{i+2}))\\
=\, &\limsup_n\sum_{\{i:[t^n_{i+1},t^n_{i+2}) \cap A = \emptyset\}}r_n(t^n_{i+1},X(^n t_{i+2}^-))\vert\vert X(t^n_{i+2})-X(t^n_{i+1})\vert\vert_2^2\\
\leq\, &C\limsup_n \max_i r_n(t^n_{i+1},X(t^{n-}_{i+2})) \sum_{\{i:[t^n_{i+1},t^n_{i+2}) \cap A = \emptyset\}} \vert\vert X(t^n_{i+2})-X(t^n_{i+1})\vert\vert_2^2\\
\leq\, &C\limsup_n \max_i r_n(t^n_{i+1},X(t^{n-}_{i+2}))[X](t) = 0 \condition{$\mathbf{P}$-a.s.}
\end{align*}

Therefore, the error term in $(2.17)$ converges to 0, and since $A$ is finite, all the terms with sums over the intervals intersecting $A$ converge to:
\begin{align*}
    \sum_{s \in A \cap [0,t)} (F(s,X_{\wedge s})-F(s,X_{\wedge s}^-))^2-Tr(\nabla F(s,X_{\wedge s^-})\nabla F(s,X_{\wedge t^-})^T \Delta X(s)\Delta X(s)^T).
\end{align*}
Taking $\epsilon \to 0$, and since $X$ has finite quadratic variation, this last sum converges to $\sum_{s \in [0,t)} (F(s,X_{\wedge s})-F(s,X_{\wedge s}^-))^2-Tr(\nabla F(s,X_{\wedge s^-})\nabla F(s,X_{\wedge t^-})^T\Delta X(s)\Delta X(s)^T)$. Therefore, by cancelling the discontinuous part in the integral from before, and adding the possible jump at $t$:

\begin{dmath}
[F(\cdot,X)](t) = \int_0^t Tr(\nabla F(s,X_{\wedge s^-})\nabla F(s,\wedge s^-)^T \mathrm{d}[X]^c(s))+\sum_{s \in [0,t]} \Delta F(s,X_{\wedge s})^2
\end{dmath},

\noindent where again $[X]^c$ is the continuous part of the quadratic variation $[X]$, i.e., $[X]^c(t) = [X](t)-\sum_{t \in [0,T]} (\Delta X(t))^2$.

Finally, since $X$ is almost surely bounded, letting $T_M := \inf\{t > 0: ||X(t)||_2>M\}$, $(2.18)$ holds locally for $X_{t\wedge T_M}$, and by taking $M \to \infty$ one obtains the general version. In conclusion, if $F \in C^{1,1}$, then $[F(\cdot,X)]_t$ is given by (2.14), and since $2[F(\cdot,X),X](t) = [F(\cdot,X)+X](t)-[F(\cdot,X)](t)-[X](t)$, then

\begin{dmath*}
[F(\cdot,X),X^j](t) = \sum_{i = 1}^d\int_0^t \partial_i F(s,X_{\wedge s})\mathrm{d}[X^j]^c(s)+\sum_{s \in [0,T]} \Delta F(s,X_{\wedge s})\Delta X^j(s)
\end{dmath*}.

\noindent Which in turn, allows us to define the Fisk-Stratonovich integral for $F \in C^{1,2}, \partial_i F \in C^{1,1}, i = 1,...,d$, as:

\begin{dmath}
\int_0^t F(s,X_{\wedge s})\circ dX^j(s) = \int_0^t F(s,X_{\wedge s}) dX^j(s)+\frac{1}{2}\sum_{i = 1}^d[\partial_i F(\cdot,X),X^j]^c(t)
\end{dmath}.
\end{proof}

At the time of the writing of these notes and after the above results were obtained, the authors came upon the preprint \cite{bouchardOctober} by Bouchard, and Vallet. There, the authors establish a decomposition for $C^{0,1}$ functionals of c\`{a}dl\`{a}g weak Dirichlet processes, from which a version of the previous theorem follows for vertical derivatives in $C^{0,1}$. To prove this result, \cite{bouchardOctober} assumes an integral condition that is more general than the horizontal differentiability needed for Theorem 2.3, but requires the uniform continuity of the functional $F$.

\section{Derivative in the Direction of a Functional} 
 
Next, a functional derivative is introduced, with the purpose of studying the behavior of a functional along general directions including, but not limited to, the horizontal one. Following this, its relationship to the horizontal derivative is established. At first, let $\gamma: [0,T] \times D([0,T],\mathbb{R}^d) \to \mathbb{R}^d$ be random (the dependency of $\gamma$, and $g$ below, with respect to the source of randomness will be omitted in the notation) non-anticipative, boundedness-preserving, and g-Lipschitz, i.e., such that for all $t \in [0,T]$, $\vert\vert \gamma(t,x)-\gamma(t,y)\vert\vert_2 \leq g(t)\vert\vert x_{\wedge t}-y_{\wedge t}\vert\vert_\infty$, with $g \geq 0$ such that $\int_0^T g(t)\,\mathrm{d}t < \infty$ a.s. Above, both $\gamma$ and $g$ are assumed to be random. Indeed, $\gamma$ depends on the random function $(X(t))_{t \in [0,T]}$ while, for further generality, $g$ could depend on $(X(t))_{t \in [0,T]}$ as well as on additional sources of randomness such as in \cite[Section V.3]{Protter}.\\

Next, for any function $w \in D([0,s],\mathbb{R}^d)$, the existence, and uniqueness, of solutions $Y^{t,w}$ to the differential equation

\begin{align}
\begin{split}
dY(t) &= \gamma(t,Y_{\wedge t})dt\\
Y(t) &= w(t) \condition{for $t \in [0,s]$},
\end{split}
\end{align}

\noindent are established. First, for the moment, assume that $s = 0$, and let,

\begin{align*}
    T_M &:= \inf\left\{t \in [s,T] : \int_s^T \vert\vert \gamma(t,0_{\wedge t})\vert\vert_2\,\mathrm{d}t \lor \int_s^T g(t)\,\mathrm{d}t \geq M\right\},
\end{align*}

\noindent (where the infimum over the empty set is taken to be $\infty$). Then, define the operator $I_M$ acting on $D([0,T],\mathbb{R}^d)$ by

\begin{align*}
    I_M(x)(t) &:= w(0) + \int_s^{t\wedge \frac{1}{2M}\wedge T_M} \gamma(s,x_{\wedge s})\,\mathrm{d}s\\
    &= w(0)+\int_0^{t\wedge \frac{1}{2M} \wedge T_M} \gamma(s,0_{\wedge s})\,\mathrm{d}s + \int_0^{t\wedge 1/2M \wedge T_M} \left(\gamma(s,x_{\wedge s})-\gamma(s,0_{\wedge s})\right)\,\mathrm{d}s.
\end{align*}

\noindent From its very definition, $I_M$ is a contraction for the infinity norm, and therefore, by the Banach fixed point theorem, there exists a unique continuous solution $x^M$ such that for $t \leq 1/2M$, $x^M(t) = w(0)+\int_0^{t\wedge T_M} \gamma(s,x^M_{\wedge s})\,\mathrm{d}s$. Now, if $s > 0$, then a new functional $\gamma_s$ with domain $D([s,T],\mathbb{R}^d)$ can be defined via $\gamma_s(t,x_{\wedge t}) := \gamma(t,\mathbbm{1}_{[0,s]}(\cdot)w(\cdot)+\mathbbm{1}_{(s,T]}(\cdot)x_{\wedge t}$ and the previous observations still allow to define solutions up to (and including) $s+1/2M$. Therefore, using increments of size $1/2M$, a unique continuous solution to

\begin{align*}
    x^M(t) &= w(s)+\int_s^{t\wedge T_M}\gamma(s,x^M_{\wedge s})\,\mathrm{d}s, \,\,\text{$t > s$},\\
    x^M(t) &= w(t), \,\,\text{$t \leq s$},
\end{align*}

\noindent can be defined. Then, since this solution is unique, letting $M \to \infty$ allows to recover a unique continuous function $x$ satisfying the original equation in $[0,T]$. Denoting solutions to $(3.1)$ by $Y^{s,w}$, leads to the following.

\begin{definition}[Derivative in the direction of a functional]
A functional $F:[0,T]\times D([0,T],\mathbb{R}^d) \to \mathbb{R}$ is said to be differentiable in the direction $\gamma:[0,T]\times D([0,T],\mathbb{R}^d) \to \mathbb{R}^d$ at $(t,x) \in [0,T) \times D([0,T),\mathbb{R}^d)$ if the following limit exists:

\begin{dmath}
D^\gamma F(t,x_{\wedge t}) = \lim_{h \to 0^+} \frac{F(t+h, Y^{t,x}_{\wedge t+h})-F(t,x_{\wedge t})}{h}
\end{dmath},

\noindent where $Y^{t,x}$ is the solution to the differential equation $(3.1)$. Moreover, $F$ is said to be differentiable in the direction $\gamma$ if it is differentiable for every pair $(t,x)$ in $[0,T) \times D([0,T),\mathbb{R}^d)$
\end{definition}

\begin{remark}
\begin{enumerate}[(i)]
\item Above, the g-Lipschitz property guarantees the existence of $Y^{s,w}$. If such a process exists for any $(s,w) \in [0,T)\times D([0,T],\mathbb{R}^d)$, then the notion of derivative just put forward can still be defined. From Theorem 2.1, it is seen that if $F \in C^{1,2}$, and if the derivatives involved are at least right-continuous, then

\begin{dmath}
D^\gamma F(t, x_{\wedge t}) = D F(t,x_{\wedge t})+\langle \nabla F(t,x_{\wedge t^-}), \gamma(t,x_{\wedge t})\rangle 
\end{dmath}.

\noindent Note that (3.3) indicates that given the existence of the space derivative, the existence of the horizontal derivative $DF$ and of, $D^\gamma F$, the derivative in the direction $\gamma$, are equivalent to one-another under some smoothness assumptions. Moreover, if $\gamma = 0$, i.e., in the flat direction, the definition of horizontal derivative is recovered. Furthermore, under the same smoothness assumptions, if the horizontal derivative is given, and the $\gamma$-derivative can be obtained for $d$ functionals $(\gamma_1, \gamma_2,...,\gamma_d)$ such that for all pairs $(t,X_{\wedge t})$, $(\gamma_1(t,X_{\wedge t}), \gamma_2(t,X_{\wedge t}),...,\gamma_d(t,X_{\wedge t}))$ are linearly independent (in particular, $\gamma_1(t,X_{\wedge t}) \neq 0$ if $d = 1$). Then, one can recover the vertical derivatives through $(3.3)$, and thus define them without the need for discontinuities.

\item The derivative along any fixed smooth path is also obtained from (3.2) by taking $\gamma$ to be equal to the derivative of the path. More precisely, if at any time $t \in [0,T]$ a derivative is defined using extensions along a fixed smooth path $y:[0,T] \to \mathbb{R}^d$ with slope $y'(t)$, then it is enough to take $\gamma(t, X_{\wedge t}) = y'(t)$. Furthermore, since $(3.3)$ is only influenced by the slope of the extension at any pair $(t,X_{\wedge t})$, one could define the derivative using a functional that defines a constant slope in which to extend the path for each of these pairs. Again, this definition is covered by Definition 3.1 by selecting a $\gamma$ that is constant along constant slopes. Finally, the ability to extending functions in non-constant directions is of interest (see Proposition 5.3)
\end{enumerate}
\end{remark}

The following relationship between horizontal and $\gamma$-derivatives holds true:

\begin{Theorem}

Let $F \in C^{0,1}$, be differentiable in the direction of a g-Lipschitz boundedness-preserving $\gamma$, and let $DF(t,x_{\wedge t}) := D^\gamma F(t,x_{\wedge t})-\langle \nabla F(t,x_{\wedge t^-}), \gamma(t,x_{\wedge t})\rangle$. Then, 

\begin{align}
    F(t+h,x_{\wedge t})-F(t,x_{\wedge t}) = \int_0^h DF(t+s, x_{\wedge t})\,\mathrm{d}s.
\end{align}

\noindent In other words, the derivative in the $\gamma$-direction allows for the construction of a horizontal derivative in the Radon-Nikodym sense. Moreover, if the right-hand side of $(3.4)$ is right differentiable, this derivative is the horizontal one.
\end{Theorem}

\begin{proof}
First, take $y \in C([0,T];\mathbb{R}^d)$, $y$ of bounded variation with $y(0) = x(0)$. Next, define $y^{k,n}$, $n \geq 1$, $k \in \{1,...,n\}$ sequentially via:

\begin{dgroup*}
\begin{dmath*}
dy^{k,n}(t) = \gamma(t,y^{k,n}_{\wedge t})dt
\end{dmath*},
\begin{dmath*}
y^{k,n}_{\wedge \frac{(k-1)T}{n}^-} = y^{k-1,n}_{\wedge \frac{(k-1)T}{n}} 
\end{dmath*},
\begin{dmath}
y^{k,n}(kT/n) = y(kT/n)
\end{dmath}.
\end{dgroup*}

\noindent More precisely, for $a \in D([0,s];\mathbb{R}^d)$, and $b \in D([0,T-s];\mathbb{R}^d)$, let

\begin{align*}
(a\otimes_s b)(t) := a(t)\mathbbm{1}_{t < s}+ b(t-s)\mathbbm{1}_{t \geq s}.
\end{align*}

Then, after defining each $y^{k-1,n}$, one defines $\gamma^{k-1,n}: [0,T-\frac{(k-1)T}{n}]\times D([0,T-\frac{(k-1)T}{n}];\mathbb{R}^d)$ via $\gamma^{k-1,n}(t,w) := \gamma(t,y^{k-1,n}\otimes_{\frac{(k-1)T}{n}} w)$. Thus $(3.5)$ turns into:

\begin{align}
\begin{split}
&dz = \gamma^{(k-1),n}(t,z_{\wedge t})\mathrm{d}t,\\
&z(1/n) = y(kT/n),
\end{split}
\end{align}

\noindent which has a solution using the same g-Lipschitz arguments as before. Then, define

\begin{align*}
    &y^{k,n} = y^{k-1,n}\otimes_{\frac{(k-1)T}{n}} z,\\
    &y^n = \sum_{k = 1}^n y^{k,n}(t)\mathbbm{1}_{[(k-1)T/n,kT/n)}(t),
\end{align*}
\noindent and take $C$ such that a.s. $\vert\vert \gamma(t,z_{\wedge t})\vert\vert_2 < C$, for all $z$ satisfying $\vert\vert z-y\vert\vert_\infty < M$. Note that solutions to $(3.6)$ have the form $z^{k,n}(t) = y(kT/n)-\int_t^{kT/n}\gamma^{k-1,n}(s,z^{k,n}_{\wedge s})\,\mathrm{d}s = y(kT/n)-\int_t^{kT/n}\gamma(s,y_{\wedge s})-\int_t^{kT/n}(\gamma^{k-1,n}(s,z^{k,n}_{\wedge s})-\gamma(s,y_{\wedge s}))\,\mathrm{d}s$. Moreover, since $y$ is uniformly continuous, there exists $N_1 \geq 1$ such that for $n \geq N_1$, if $|t-s| < 1/n$, then $\vert y(s)-y(t)\vert < \epsilon$. Furthermore for any $\epsilon > 0$, there also exists $N_2 \geq 1$ such that for any interval $I$ of length less than $1/N_2$, $\int_I g(t)\,\mathrm{d}t < \epsilon$. Take $N = N_1 \lor N_2$, and note that for $k = 1$,

\begin{align*}
    \vert\vert z^{1,n}(t)-y(t)\vert\vert_2 \leq \epsilon+C/n+\vert\vert z^{1,n}-y\vert\vert_\infty \epsilon,
\end{align*}
\noindent implies that,
\begin{align*}
    \vert\vert z^{1,n}-y\vert\vert_\infty(1-\epsilon) \leq \epsilon+C/n,
\end{align*}

\noindent and that for $\epsilon$ small enough this implies $\vert\vert z^{1,n}-y\vert\vert_\infty < M$. Assume, for the induction hypothesis, that this is true for $k =1,...,m$, then for $m + 1$,

\begin{align*}
    \vert\vert z^{1,m+1}-y\vert\vert_\infty &\leq \epsilon+C/n+(M \lor \vert\vert z^{1,m+1}-y\vert\vert_\infty)\int_I g(t)\,\mathrm{d}t.
\end{align*}

Regardless of which of the two values the maximum may take, by taking $\epsilon$ small enough this implies $\vert\vert z^{1,m+1}-y\vert\vert_\infty < M$, which in turn implies $\gamma^{k-1,n}(t,y^n_{\wedge t}) < C$. This common bound can then be used to see that $\vert\vert y-y^n\vert\vert_\infty < \epsilon+C/n+2C\epsilon \to 0$, uniformly in $n$.

\noindent Finally, define $f^{k,n}:\mathbb{R}^d \to \mathbb{R}$, via $f^{k,n}(z) := F(kT/n,y_{\wedge kT/n}^{z-y(kT/n)})$, and note that $\nabla f^{k,n}(z) = \nabla F(kT/n,y_{\wedge kT/n}^{z-y(kT/n)})$. Next,

\begin{align*}
    F(kT/n&,y^{k,n}_{\wedge \frac{kT}{n}})-F((k-1)T/n,y^{k-1,n}_{\wedge (k-1)T/n})\\
    =  &F(kT/n,y^{k,n}_{\wedge \frac{kT}{n}^-})-F((k-1)T/n,y^{k-1,n}_{\wedge (k-1)T/n})\\
    &+f^{k,n}\left(y((k+1)T/n)-\int_{Tk/n}^{T(k+1)/n} \gamma(t,y^{k+1,n}_{\wedge t})\,\mathrm{d}t\right)-f^{k,n}(y(kT/n))\\
    = &\int_{(k-1)T/n}^{kT/n} D^\gamma F(t,y^{k,n}_{\wedge t})\,\mathrm{d}t
    +\int_{kT/n}^{(k+1)T/n} \nabla F(kT/n,y_{\wedge kT/n}^{z-y(kT/n)})\,\mathrm{d}y(x)\\
    &-\int_{kT/n}^{(k+1)T/n} \langle \nabla F(kT/n,y_{\wedge kT/n}^{z-y(kT/n)}),\gamma(t,y_{\wedge t})\rangle\,\mathrm{d}t.
\end{align*}

\noindent Since $F(T,y^{n,n}_{\wedge T}) \to F(T,y_{\wedge T})$, and since

\begin{dmath*}
F(T,y^{n,n}_{\wedge T})-F(0,y^{0,n}_{\wedge 0}) = \sum_{k = 1}^n  \left(F(kT/n,y^{k,n}_{\wedge \frac{kT}{n}})-F((k-1)T/n,y^{k-1,n}_{\wedge (k-1)T/n})\right),
\end{dmath*}

\noindent it follows that:

\begin{align}
    F(T&,y^{n,n}_{\wedge T})-F(T/n,y^{1,n}_{\wedge T/n}) \nonumber\\
    = &\int_{0}^T \sum_{k = 2}^n D^\gamma F(t,y^{k,n}_{\wedge t})\mathbbm{1}_{[(k-1)T/n,kT/n)}(t)\,\mathrm{d}t\nonumber\\
    &+\int_{T/n}^T \sum_{k = 2}^{n-1} \nabla F(kT/n,y_{\wedge kT/n}^{z-y(kT/n)})\mathbbm{1}_{[kT/n,(k+1)T/n)}(t)\,\mathrm{d}y(z)\nonumber\\
    &+\int_{T/n}^T \sum_{k = 2}^{n-1} \langle\nabla F(kT/n,y_{\wedge kT/n}^{z-y(kT/n)}),\gamma(t,y^{k+1,n}_{\wedge t})\rangle\mathbbm{1}_{[kT/n,(k+1)T/n)}(t)\,\mathrm{d}t.
\end{align}

Since $y_{k,n}(t) \to y(t)$ uniformly in $t$, the $\gamma$-derivative is fixed-times continuous and the space derivatives are left-continuous, by taking the limit as $n \to \infty$, the identity $(3.6)$ turns into:

\begin{dmath*}
 F(T,y_{\wedge T})-F(0,y_{\wedge 0}) = \int_{0}^T D^\gamma F(t,y^{k,n}_{\wedge t})\,\mathrm{d}t+\int_0^T \nabla F(t,y_{\wedge t^-})\,\mathrm{d}y(t)+\int_0^T \langle\nabla F(t,y_{\wedge t^-}),\gamma(t,y_{\wedge t^-})\rangle\,\mathrm{d}t.
\end{dmath*}

If the time $s$ at which the function is to be extended is different from $t = 0$, it is enough to redefine $\bar{y} = x\otimes_s y$, and $F_s:[0,T-s]\times D([0,T-s];\mathbb{R}^d)$ such that $F_s(t,w) = F(t+s, x\otimes_s w)$ and the conclusion of the above theorem continues to hold. Thus under the existence of a continuous at fixed times $\gamma$-derivative and a left-continuous, in time, space derivative, all boundedness-preserving, the following relationship is obtained:

\begin{dmath}
F(t+h,x_{\wedge t})-F(t,x_{\wedge t}) = \int_t^{t+h} D^\gamma F(s,x_{\wedge t})\,\mathrm{d}s-\int_t^{t+h}\langle \nabla F(s,x_{\wedge t}, \gamma(s,x_{\wedge t^-}))\,\mathrm{d}s.
\end{dmath}
\end{proof}

Note that $(3.8)$ allows to write extensions along fixed paths as an absolutely continuous function with respect to the Lebesgue measure, which is the property used in the proof of the functional It\^{o} formula. Moreover, if $D^\gamma, \nabla F$, and $\gamma$ are right continuous then the horizontal derivative exists and is given by $(3.3)$\\

\section{Functional It\^{o} Formula for L\'{e}vy Processes}

The main focus of this section is on functionals of L\'{e}vy processes, starting with processes driven by L\'{e}vy type integrals with path dependent coefficients. A few authors have previously derived It\^{o} type formulas for this case e.g., \cite{MemoryJumps}, and \cite{LEVENTAL13}, where the former deals with functionals of functions in $L^p$, while the later recovers a functional It\^{o} formula for the case of a general semimartingale. Additionally, \cite{ConvexRisk} studies L\'{e}vy processes with finite second moment, and applies a functional It\^{o} formula to the case of the exponential process.\\

From now on, we deal with measures $\mathbf{P}$ for which almost surely, the process 
\begin{align*}
    (X(t))_{t \in [0,T]} = (X^1(t),X^2(t),..,X^d(t))_{t \in [0,T]},
\end{align*}
\noindent can be written as a  L\'{e}vy type integral, e.g., \cite{applebaum_2009}:

\begin{multline}
X^i(t)-X^i(0) = \int_0^t G^i(s)\,\mathrm{d}t+\int_0^t L^i_j(s)\mathrm{d}B^j(s)
+\int_0^t\int_{\{||x||_2 > 1\}} K^i(t,x)\,N(dt,dx)\\
+\int_0^t\int_{\{||x||_2 \leq 1\}} H^i(t,x)\,\tilde{N}(dt,dx) \,\,\text{$\mathbf{P}$-a.s.}
\end{multline}

Above, under the probability measure $\mathbf{P}$, $B = (B^1,...,B^m)$, $m \leq d$, is a multidimensional Brownian motion with independent components, $N$ is a L\'{e}vy process with triplet $(0,0,\nu)$, $\tilde{N}$ its corresponding compensated process, and both the vector $G = (G^1,...,G^d)$ and the matrix $(L^i_j)_{1\leq i\leq d, 1\leq j\leq m}$ have predictable entries, all adapted to $(\mathcal{F}_t)_{t \in [0,T]}$, satisfying:

\begin{dmath*}
P\left(\int_0^T \vert|G(s)\vert\vert_1\,\mathrm{d}s < \infty \right) = P\left(\int_0^T \vert\vert L(s)\vert\vert_F^2\,\mathrm{d}s < \infty\right) = 1
\end{dmath*},

\noindent where $\vert\vert\cdot\vert\vert_1$ is the usual $\ell_1$-norm, and $\vert\vert\cdot\vert\vert_F$ is the Frobenius norm of a matrix. Similarly, $K = (K^1,...,K^d),H = (H^1,...,H^d)$ are predictable $\mathcal{F}_t \otimes \mathcal{B}(\mathbb{R}^d)$-adapted processes, with $H$ such that:

\begin{dmath*}
P\left(\int_0^T\int_{\mathbb{R}^d \setminus \{0\}} \vert\vert H(t,x)\vert\vert_2^2\,\nu(dx)\mathrm{d}t < \infty \right)=1.
\end{dmath*}

Once $(X(t))_{t \in [0,T]}$ is defined, one can apply the functional It\^{o} formula for semimartingales from \cite{Barcelona}, to obtain that if $F$ is a $C^{1,2}$-functional, then $\mathbf{P}$-a.s.:

\begin{dmath}
F(T,X_{\wedge T})-F(0,X_{\wedge 0}) = \int_0^T DF(t,X_{\wedge t})\,\mathrm{d}t+\int_0^T \nabla F(t,X_{\wedge t^-})\cdot\mathrm{d}X(t)+\frac{1}{2}\int_0^T Tr(\nabla^2 F(t,X_{\wedge t}) \mathrm{d}[X]^c(t))\\
+\sum_{t \in [0,T]} \left(F(t,X_{\wedge t})-F(t,X_{\wedge t^-})-\langle\nabla F(t,X_{\wedge t^-}),\Delta X(t)\rangle\right).
\end{dmath}

Moreover, since elements in $[0,T]\otimes D([0,T], \mathbb{R}^d)$ can be seen as triplets in $[0,T]\otimes D([0,T]);\mathbb{R}^d)\otimes \mathbb{R}^d$  such that $(t,x) \simeq (t,x_{\wedge t^-},x(t))$, and the sample space in the current setting is given by $D([0,T],\mathbb{R}^d)$, then classical arguments such as those in \cite[Section 4]{applebaum_2009} ensure that if

\begin{dmath}
\sup_{t \in [0,T]}\sup_{||x||_2 \leq 1} \vert\vert H(t,x)\vert\vert_2 < \infty,
\end{dmath}

 \noindent then $(4.2)$ can be written in a way that showcases in a more direct manner the components of the L\'{e}vy integral process:

\begin{align}
F(T,X_{\wedge T})-F(0,X_{\wedge 0}) = &\int_0^T DF(t,X_{\wedge t})\,\mathbb{d}t+\int_0^T\langle \nabla F(t,X_{\wedge t^-}),G(t)\rangle\,\mathrm{d}t+\int_0^T \nabla F(t,X_{\wedge t^-})^TL(t)\,\mathrm{d}B(t)\nonumber\\
&+\frac{1}{2}\int_0^T Tr(\nabla^2 F(t,X_{\wedge t}) L(t)L^T(t))\,\mathrm{d}t+\int_0^T\int_{||x||_2 > 1} \left(F(t,X_{\wedge t^-}^{K(t,x)})-F(t,X_{\wedge t^-})\right)\,N(dt,dx)\nonumber\\
&+\int_0^T\int_{||x||_2 \leq 1} \left(F(t,X_{\wedge t^-}^{H(t,x)})-F(t,X_{\wedge t^-})\right)\,\tilde{N}(dt,dx)\nonumber\\
&+\int_0^T \int_{||x||_2 \leq 1}\left(F(t,X_{\wedge t^-}^{H(t,x)})-F(t,X_{\wedge t}^-)-\langle \nabla F(t,X_{\wedge t^-}),H(t,x)\rangle\right) \,\nu(dx)\mathrm{d}t \condition{$\mathbf{P}$-a.s.}
\end{align}

If the measure $\mathbf{P}$ is such that $(X(t))_{t \in [0,T]}$ is a multivariate L\'{e}vy process with triplet $(\mu,\Sigma,\nu)$, then:

\begin{dmath}
X(t) = \mu t+\Sigma^{1/2} B(t)+\int_0^t\int_{||x||_2 > 1} x\, N(dt,dx)+\int_0^t\int_{||x||_2 \leq 1} x\, \tilde{N}(dt,dx),
\end{dmath}

\noindent where, if $m:= rank(\Sigma)$, $B$ is an $m$-dimensional standard Brownian motion, and $\Sigma^{1/2}$ is a $d\times m$ matrix such that $\Sigma = \Sigma^{1/2}(\Sigma^{1/2})^T$. From Theorem \ref{Functional Ito Formula}, if $F \in C^{1,2}$ it follows that:

\begin{align}
F(T&,X_{\wedge T})-F(0,X_{\wedge 0})\nonumber\\
= &\int_0^T DF(t,X_{\wedge t})\,\mathrm{d}t+\int_0^T \langle \nabla F(t,X_{\wedge t}),\mu\rangle\,\mathrm{d}t+\int_0^T  \nabla F(t,X_{\wedge t^-})^T \Sigma^{1/2}\,\mathrm{d}B(t)\nonumber\\
&+\frac{1}{2}\int_0^T Tr(\nabla^2 F(t,X_{\wedge t}) \Sigma)\,\mathrm{d}t
+\int_0^T\int_{||x||_2 > 1} \left(F(t,X_{\wedge t^-}^x)-F(t,X_{\wedge t^-})\right)\,N(dt,dx)\nonumber\\
&+\int_0^T\int_{||x||_2 \leq 1} \left(F(t,X_{\wedge t^-}^x)-F(t,X_{\wedge t^-})\right)\,\tilde{N}(dt,dx)\nonumber\\
&+\int_0^T\int_{||x||_2 \leq 1} \left(F(t,X_{\wedge t^-}^x)-F(t,X_{\wedge t^-})-\langle\nabla F(t,X_{\wedge t^-}),x\rangle\right) \, \nu(dx)\mathrm{d}t.
\end{align}

The main objective of the forthcoming results is to relax the vertical differentiabilty conditions on $F$ and the convergence requirements in the last integral from $(4.6)$, using functional analogues of the argument in \cite{Eisenbaum09}. With this objective in mind, the following definition is recalled.

\begin{definition}[Integral with respect to local time]

Let $F:\mathbb{R}^d \to \mathbb{R}$, and let $(B(t))_{t \in [0,T]}$ be a standard Brownian motion. The integral with respect to the local time measure $dL^x_s(B)$ is given by:

\begin{align*}
    \int_0^t\int_{\mathbb{R}}f(x(s)|_{x^j(s) = x})\,\mathrm{d}L^x_s(B) &= \int_0^t f(x(s)|_{x^j(s) = B(s)})\,\mathrm{d}B(s)+\int_{T-t}^T f(\hat{x}(s)|_{\hat{x}^j(s) = \hat{B}(s)})\,\mathrm{d}\hat{B}(s),
\end{align*}

\noindent where for any $x = (x^1,...,x^d) \in \mathbb{R}^d$, $x|_{x^j = y} = (x^1,...,x^{j-1},y,x^{j+1},...,x^d)$, and\\
$\hat{x}(t) = (\hat{x}^1(t),...,\hat{x}^d(t)) = (x^1(T-t),...,x^d(T-t))$.
\end{definition}

If $(B(t))_{t \in [0,T]}$ is a multivariate Brownian motion with independent components, and if the process $(N(t))_{t\in[0,T]}$ is independent of $(B(t))_{t \in [0,T]}$, as in \cite[Section 6]{EISENBAUM} by conditioning with respect to these processes, it can be seen that the following two properties hold true:

\begin{enumerate}[i)]
    \item \begin{align}
    \mathbb{E}\left\vert\int_0^t\int_{\mathbb{R}} f(s,B_s|_{B_s^i = x},N_s)\,\mathrm{d}L^x_s(B^i)\right\vert \leq \vert\vert f\vert\vert_{L,i},
    \end{align}
    \noindent where
\begin{align*}
    \vert\vert f\vert\vert_{L,i} &= 2\mathbb{E}\left[\int_0^T f^2(t,B(t),N(t))\,\mathrm{d}t\right]^{1/2}+\mathbb{E}\int_0^T \frac{\vert f(t,B(t),N(t))B^i(t)\vert}{t}\,\mathrm{d}t.
\end{align*}
\item\begin{align}
    \int_0^t\int_{\mathbb{R}} f(s,B_s|_{B_s^i = x},N_t)\,\mathrm{d}L^x_s(B^i) = -\int_0^t \frac{\partial f}{\partial x^i}(s,B_s,N_s) \,\mathrm{d}s.
\end{align}
\end{enumerate}

Note that given any L\'{e}vy process $(X(t))_{t \in [0,T]}$, with triplet $(\mu,\Sigma,\nu)$, and $B$ as its Brownian component as in $(4.5)$, then $(B(t))_{t \in [0,T]}$ is $\sigma((X(t))_{t \in [0,T]})$ measurable, and thus $(N(t))_{t \in [0,T]}$ such that $N(t) = X(t)-\Sigma^{1/2}B(t)$ is measurable as well. Moreover, given a differentiable function $f: \mathbb{R}^d \to \mathbb{R}$, then $\tilde{f}: \mathbb{R}^m\to \mathbb{R}$ given by $\tilde{f}(x) := f(\Sigma^{1/2}x+N(t))$ is such that $\tilde{f}(B(t)) = f(X(t))$, and clearly $\nabla \tilde{f} = (\Sigma^{1/2})^T \nabla f$. Having established notation, the following operators, where $C(\mathbb{R}^m,\mathbb{R})$ is the set of continuous functions from $\mathbb{R}^m$ to $\mathbb{R}$ and where $R = ((\Sigma^{1/2})^T\Sigma^{1/2})^{-1}(\Sigma^{1/2})^T$, can be defined.

\begin{definition}
For $F \in C(\mathbb{R}^m,\mathbb{R})$, and for $i \in \{1,...,m\}$, let $I_i: C(\mathbb{R}^m,\mathbb{R}) \to C(\mathbb{R}^m,\mathbb{R})$, $\mathcal{A}_i: C(\mathbb{R}^m,\mathbb{R})\to C(\mathbb{R}^m,\mathbb{R})$, and $\mathcal{L}: C(\mathbb{R}^m,\mathbb{R}) \to C^{0,0}$ be defined via.

\begin{enumerate}
    \item $I_iF(x) := \int_0^{x^i} F(x|_{x^i = y})\,\mathrm{d}y$,\\
    \item $IF(x) := (I_1 F(x),...,I_mF(x))$,\\
    \item $\mathcal{A}_iF(x) :=  \frac{\partial^2F}{\partial x_i^2} (x)+\int_{\vert\vert y\vert\vert_2 \leq 1}\int_0^1\left(\frac{\partial F}{\partial x_i}(x+sRy)-\frac{\partial F}{\partial x_i}(x)\right)(Ry)_i\,\nu(\mathrm{d}y)$,\\
    \item $\mathcal{A}F(x) := (\mathcal{A}_1F_1(x),...,\mathcal{A}_mF_m(x))$,\\
    \item $\mathcal{L}_tF(B_{\wedge t}) := \sum_{i = 1}^m \int_0^t\int_{\mathbb{R}} F_i(B(t)|_{B^i(t) = x})\,\mathrm{d}L^x_s(B^i)$.
\end{enumerate}
\end{definition}

Next, we present an extension of \cite[Theorem 1.1]{Eisenbaum09} to the multivariate case:

\begin{Theorem}
Let $(X(t))_{t \in [0,T]}$ be a multivariate L\'{e}vy Process with triplet $(\mu,\Sigma,\nu)$, let $Q: \mathbb{R}^d\to \mathbb{R}^d$ be the orthogonal projection onto the range of $\Sigma^{1/2}$, let $F:\mathbb{R}^d \to \mathbb{R}$ be continuously differentiable, and let 

\begin{align*}
    \int_{\vert\vert x\vert\vert_2 \leq 1} \vert\vert (Id-Q) x\vert\vert_2\,\nu(\mathrm{d}x) < \infty.
\end{align*} 
\noindent Then,

\begin{align}
    F(X(t))-F(0) = &\int_0^t \nabla F(X(s))^T \Sigma^{1/2}\cdot\mathrm{d}B(s)+\int_0^t \langle \nabla F(X(s^-)),\mu\rangle\,\mathrm{d}s\nonumber\\
    &+\int_0^t\int_{\vert\vert y\vert\vert_2 \leq 1} \left(F(X(s^-+y))-F(X(s^-))\right)\,\mathrm{d}\tilde{N}(s,y)\nonumber\\
    -\mathcal{L}_t\mathcal{A}I\tilde{F}(t,X_{\wedge t}))\nonumber\\
    &+\int_0^t\int_{\vert\vert y\vert\vert \leq 1}\left(F(X(s^-+y))-F(X(s^-+Qy))
    -\langle\nabla F(X(s^-)),(I-Q)y\rangle\right)\,\nu(\mathrm{d}y)\mathrm{d}s.
\end{align}
\end{Theorem}

\begin{proof}
As in \cite[Theorem 1.1]{Eisenbaum09}, let $(\phi_n)_{n \geq 1}$, $\phi_n: \mathbb{R}^d\to \mathbb{R}$, be a sequence of mollifiers. Then $F_n := F\ast \phi_n$, and $\tilde{F}_n$, are sequences of continuously differentiable, locally bounded, functions such that $\lim_{n \to \infty} F_n(x) = F(x)$, $\lim_{n \to \infty} \tilde{F}_n(x) = \tilde{F}(x)$, $\lim_{n \to \infty} \nabla F_n(x) = \nabla F(x)$, and $\lim_{n \to \infty} \nabla \tilde{F}_n = \nabla \tilde{F}_n(x)$. Applying It\^{o}'s formula to the approximations $F_n$ gives:

\begin{align}
    F_n(X(t))-F_n(0) = &\int_0^t \nabla F_n(X(t))\cdot\mathrm{d}B(s) +\int_0^t  \langle \nabla F_n(X(s)),\mu\rangle\,\mathrm{d}s+\frac{1}{2}\int_0^t Tr(\nabla^2 F_n(X(s))\Sigma)\,\mathrm{d}s\nonumber\\
    &+\int_0^t \int_{\vert\vert y\vert\vert_2 \leq 1} \left(F_n(X(s^-+y)-F_n(X(s^-))\right) \,\tilde{N}(ds,dy)\nonumber\\
    &+\sum_{s \in [0,t]}\left(F_n(X(s))-F_n(X(s^-))\right)\mathbbm{1}_{\{\vert\vert \Delta X(s)\vert\vert_2 > 1\}}\nonumber\\
    &+\int_0^t\int_{\vert\vert y\vert\vert_2 \leq 1} \left(F_n(X(s^-)+y)-F_n(X(s^-))-\langle \nabla F_n(X(s^-)),y\rangle\right) \,\nu(dy)\mathrm{d}t.
\end{align}

Next, as $n \to \infty$, $F_n(X(t))-F_n(0) \to F(X(t))-F(0)$. Then, since $F$, and $\nabla F$ are continuous, the arguments in \cite{EISENBAUM} allow to conclude that the first four terms on the right-hand side of $(4.10)$ all converge to a corresponding term in $(4.9)$. For the fifth and sixth term, using a stopping time argument, $B$ will be assumed to be bounded a.s. Next, as in \cite{Eisenbaum09}, the convergence of the second derivative term depends only on properties $(4.7)$, and $(4.8)$, both still satisfied in the multivariate case, and therefore,

\begin{dmath*}
    \lim_{n \to \infty} \sum_{j = 1}^m\int_0^t \partial_j^2\tilde{F}_n(B_s)\,\mathrm{d}s =  -\mathcal{L}_t \left(\frac{\partial^2 I_1F_n}{\partial x_1^2},...,\frac{\partial^2 I_mF_n}{\partial x_m^2}\right)(B_{\wedge t})
\end{dmath*}.

\noindent Next,

\begin{align*}
    &\int_0^t\int_{\vert\vert y\vert\vert_2 \leq 1} \left(F_n(X_{s^-}+Qy)-F_n(X_{s^-})-\langle \nabla F_n(X_{s^-}),Qy\rangle\right) \,\nu(dy)\mathrm{d}t\\
    &= \int_0^t\int_{\vert\vert y\vert\vert_2 \leq 1} \left(\tilde{F}_n(B_s+Ry)-\tilde{F}_n(B_s)-\langle \nabla \tilde{F}_n(B_s),Ry\rangle\right) \,\nu(dy)\mathrm{d}t\\
    &=\int_0^t\int_{\vert\vert y\vert\vert_2 \leq 1}\int_0^1 \left(\langle \nabla \tilde{F}_n(B_s+sRy),Ry\rangle-\langle\nabla \tilde{F}_n(B_s),Ry\rangle\right)\,\mathrm{d}s\nu(dy)\mathrm{d}t.
\end{align*}

\noindent Then, define

\begin{align}
    H^i_n(B_s|_{B^i_s = x}) := \int_0^x \int_{\vert\vert y\vert\vert_2 \leq 1}\int_0^1\left(\frac{\partial \tilde{F}_n}{\partial x_i}(B(s)|_{B^i(s) = z}+sRy)(Ry)_i-\frac{\partial \tilde{F}_n}{\partial x_i}(B(s)|_{B^i_s = z})(Ry)_i\right)\,\mathrm{d}s\nu(\mathrm{d}y)\mathrm{d}z,
\end{align}

\noindent thus as in \cite{Eisenbaum09}:

\begin{align*}
    -\int_0^t&\int_{\mathbb{R}}H^i_n(s,B_s|_{B^i_s = x})\,\mathrm{d}L^x_s(B^i)= \int_0^t\int_{\vert\vert y\vert\vert_2 \leq 1}\int_0^1\left(\frac{\partial \tilde{F}_n}{\partial x_i}(B_s+sRy)(Ry)_i-\frac{\partial \tilde{F}_n}{\partial x_i}(B_s)(Ry)_i\right)\,\mathrm{d}s\nu(\mathrm{d}y).
\end{align*}

Since $F$ is continuous, since $B$ can be taken to be bounded using a stopping time, and since the integrand in the right hand side of $(4.11)$ is restricted to $\vert\vert y\vert\vert_2 \leq 1$, the expression inside the last integral is bounded. Moreover, by Fubini's Theorem:

\begin{align*}
    H^i_n(B_s|_{B^i_s = x}) = &\int_{\vert\vert y\vert\vert_2 \leq 1}\int_0^x \left(\tilde{F}_n(B_s|_{B^i_s = z}+Ry)-\tilde{F}_n(B_s|_{B^i_s = z})-\frac{\partial \tilde{F}_n}{\partial x_i}(B_s|_{B^i_s = z})(Ry)_i\right)\,\mathrm{d}z\mathrm{d}s\nu(\mathrm{d}y)\\
    = &\int_{\vert\vert y\vert\vert_2 \leq 1}\int_0^{x+(Ry)_i}\tilde{F}_n(B_s|_{B^i_s = z}+(Ry)|_{(Ry)_i = 0})\,\mathrm{d}z-\int_0^x \tilde{F}_n(B_s|_{B^i_s = z})\,\mathrm{d}z\\
    &-(Ry)_i\tilde{F}_n(B_s|_{B^i_s = x})+(Ry)_i\tilde{F}_n(B_{s^-}|_{B^i_s = 0})-\int_0^{(Ry)_i}\tilde{F}_n(B_s|_{B^i_s = z}+Ry|_{(Ry_i = 0)})\mathrm{d}z\,\mathrm{d}s\nu(\mathrm{d}y).
\end{align*}

Thus, if $G^i_n(x) = \int_{\vert\vert y\vert\vert_2 \leq 1} \left(I_i \tilde{F}_n(x+Ry)-I_i \tilde{F}_n(x)-(Ry)_i\tilde{F}_n(x)\right)\,\nu(\mathrm{d}y)$, the following identity remains true.

\begin{align*}
   \int_0^t\int_{\mathbb{R}}H^i_n(B_s|_{B^i_s = x})\,\mathrm{d}L^x(B^i) &= \int_0^t\int_{\mathbb{R}}G^i_n(B_s|_{B^i_s = x})\,\mathrm{d}L^x(B^i). 
\end{align*}

Then, the limit can be taken as in the single variable case, since the derivatives are continuous. Finally, the only term whose convergence has not been verified in $(4.10)$ is:
\begin{align}
\begin{split}
    &\int_{\vert\vert y\vert\vert_2 \leq 1} \left(F_n(X(t^-)+y)-F_n(X(t^-)+Qy)-\langle\nabla F_n(X(t^-)),(I-Q)y\rangle\right)\,\nu(\mathrm{d}y)\\
    =&\int_{\vert\vert y\vert\vert_2 \leq 1}\int_0^1 \left(\langle\nabla F_n(X(t^-)+s(I-Q)y+Qy),(I-Q)y\rangle-\langle\nabla F_n(X(t^-)),(I-Q)y\rangle\right)\,\mathrm{d}s\nu\mathrm{d}y.
\end{split}
\end{align}

Since it can be assumed that $\vert\vert X(t)\vert\vert_\infty < M$ a.s., and since the integrand is restricted to $\vert\vert y\vert\vert_2 \leq 1$ while the derivatives are continuous, there exists $C >0$ such that $\vert\vert \nabla F_n(X(t)+s(I-Q)y+Qy)\vert\vert_\infty < C$, for all $s \in [0,1]$. Therefore, the last term in $(4.10)$ is dominated by:

\begin{align*}
    2C\int_{\vert\vert y\vert\vert_2 \leq 1} \vert\vert (I-Q)y\vert\vert_2 \,\nu(\mathrm{d}y) &< \infty.
\end{align*}

\noindent By dominated convergence, this last term converges to:

\begin{align*}
    \int_0^t \int_{\vert\vert y\vert\vert_2 \leq 1}\left(F(X_{s^-}+y)-F(X_{s^-}+Qy)-\langle\nabla F(X_{s^-}),(I-Q)y\rangle\right)\,\nu(\mathrm{d}y)\,\mathrm{d}s.
\end{align*}.

\noindent Therefore, all the terms in $(4.10)$ converge to their corresponding term in $(4.9)$, concluding the proof.
\end{proof}

The following theorem provides a functional analogue of the above result. It is optimal when the Gaussian component is non-degenerate, since it identifies
\begin{align*}
    F(t,X_{\wedge t})-F(0,X_{\wedge 0})-\int_0^t DF(s,X_{\wedge s})\,\mathrm{d}s-\int_0^t \langle \nabla F(s,X_{\wedge s}),\mu\rangle\,\mathrm{d}s-\int_0^t \nabla F(s,X_{\wedge s^-})^T\Sigma^{1/2}\,\mathrm{d}B(s)
\end{align*}
\noindent with an expression that makes no extra assumptions on $F$.

\begin{Theorem}[Optimal Functional It\^{o} Formula]

Let $F \in C^{1,1}$, and let $ X = (X(t))_{t \in [0,T]}$ be a L\'{e}vy process with triplet $(\mu,\Sigma,\nu)$. Let, $Q$, the orthogonal projection onto the range of $\Sigma^{1/2}$ be such that $\int_{\vert\vert y\vert\vert_2 \leq 1}\vert\vert (I-Q)y\vert\vert_2\,\nu(\mathrm{d}y) < \infty$, and let $G(t,x) := F(t,x_{\wedge t^-}^{x-x(t^-)})$. Then,

\begin{align}
F(t,X_{\wedge t})-F(0,X_{\wedge 0}) = &\int_0^t DF(s,X_{\wedge s})\,\mathrm{d}s+\int_0^t \langle \nabla F(s,X_{\wedge s}),\mu\rangle\,\mathrm{d}s+\int_0^t \nabla F(s,X_{\wedge s^-})^T\Sigma^{1/2}\,\mathrm{d}B(s)\nonumber\\
&+ \sum_{s \leq t} (F(s,X_{\wedge s})-F(s,X_{\wedge s^-}))\mathbbm{1}_{\{\vert\vert\Delta X(s)\vert\vert_2 > 1\}}\nonumber\\
&+\int_0^t \int_{||y||_2 \leq 1} \left(F(s,X_{\wedge s^-}^y)-F(s,X_{\wedge s^-})\right)\,\mathrm{d}\tilde{N}(ds,dy)-\mathcal{L}_t\mathcal{A}I\tilde{G}(B_{\wedge t})\nonumber\\
&+\int_0^t\int_{\vert\vert y\vert\vert_2 \leq 1} \left(F(s,X_{\wedge s^-}^y)-F(s,X_{\wedge s^-}^{Qy})-\langle \nabla F(t,X_{\wedge t^-}),(I-Q)y\rangle\right)\,\nu(\mathrm{d}y)\,\mathrm{d}s.
\end{align}
\end{Theorem}

\begin{proof}
Let $\tau = \{\tau_n\}_{n \geq 1}$ be a nested sequence of partitions given by stopping times $\tau_n = (t_0^n,...,t^n_{k_n})$, as in the proof of Theorem 2.2. Then, define:

\begin{align*}
X^n(t) := \sum_{i = 0}^{k_n-1} X(t_{i+1}^n)\mathbbm{1}_{[t^n_i,t^n_{i+1})}(t)+X(T)\mathbbm{1}_{\{T\}},
\end{align*}

\noindent together with

\begin{align*}
    F^n_i(x) = F\left(t,(X^n_{\wedge t^{n-}_i})^{x-X^n(t^{n-}_i)}\right).
\end{align*}

\noindent From the construction of $\tau$, $X^{n}(t)$ converges to $X(t^-)$, except at the jump times of $X$. However, as this set has Lebesgue measure $0$, then $\vert\vert X^n(t)-X(t^-)\vert\vert_\infty = \esssup_{t \in [0,T]} \vert\vert X_n(t)-X(t)\vert\vert_2 \to 0$. Since,  $F(\cdot,X_{\wedge \cdot})$ is bounded, and since $F$ is left continuous with respect to $d_*$, Lemma 2.2 ensures that $\vert\vert F(t,X^n_{\wedge t})-F(t,X_{\wedge t})\vert\vert_\infty \to 0$. The same applies to the space derivatives of $F$, ensuring $\vert\vert\nabla F(t,X^n_{\wedge t})-\nabla F(t,X^n_{\wedge t})\vert\vert_\infty \to 0$. Next, from Theorem 4.1:

\begin{align*}
    F(T,X^n_{\wedge T})&-F(0,X^n(0))= \sum_{i = 0}^{k_n-1} \left(F(t_{i+1},X^n_{\wedge t_{i+1}})-F(t_i,X^n_{\wedge t_i})\right)\\
    = &\sum_{i = 0}^{k_n-1} \left(F(t_{i+1},X^n_{\wedge t_{ i+1}})-F(t_{i+1},X^n_{\wedge t_{i+1}^-})\right)+\sum_{i = 0}^{k_n-1} \left(F(t_{i+1},X^n_{\wedge t_{i+1}^-})-F(t_i,X^n_{\wedge t_i})\right)\\
    = &\sum_{i = 0}^{k_n-1} \left(F_{i+1}^n(X(t_{i+2}^-))-F_{i+1}^n(X(t_{i+1}^-))\right)+\int_{t_i}^{t_{i+1}}D F(t,X^n_{\wedge t_i})\,\mathrm{d}t\\
    = &\sum_{i = 0}^{k_n-1} \left(\int_{t_i}^{t_{i+1}} DF(t,X_{\wedge t_i}^n)\,\mathrm{d}t+\int_{t_{i+1}}^{t_{i+2}} \nabla F_{i+1}^n(X( t^-))^T\Sigma^{1/2}\mathrm{d}B(t)\right)\\
    &+\int_{[t_{i+1},t_{i+2})}\int_{||y||_2 \leq 1} \left(F_{i+1}^n(X^n(t^-)+x)-F_{i+1}^n(X^n(t^-))\right)\,\mathrm{d}\tilde{
    N}(t,x)\\
    &+\sum_{s \in [t_{i+1},t_{i+2})} \left(F^n_{i+1}(X^n(s))-F_{i+1}^n(X^n(s^-))\right)\mathbbm{1}_{||\Delta x(s)||_2 > 1}-\mathcal{L}_T\mathcal{A}I\tilde{F}^n_{i+1}\mathbbm{1}_{[t_i,t_{i+1})}(X_{\wedge T})\\
    = &\int_0^T D F(t,X_{\wedge t}^n)\,\mathrm{d}t
    +\int_0^T \sum_{i = 1}^{k_n} \nabla F_i^n(X( t^-))^T\Sigma^{1/2}\mathbbm{1}_{[t_i,t_{i+1})}(t)\cdot\mathrm{d}B(t)\\
    &+ \int_0^T \int_{||y||_2 \leq 1} \sum_{i =1}^{k_n} (F_i^n(X^n(t^-)+y)-F_i^n(X^n(t^-)))\mathbbm{1}_{[t_i,t_{i+1})}(t)\,\mathrm{d}\tilde{N}(t,y)\\
    &+ \sum_{t \in [0,T]} \sum_{i = 1}^{k_n} (F^n_i(X^n(t)))-F_i^n(X^n(t^-)))\mathbbm{1}_{||\Delta x(t)||_2 > 1}\mathbbm{1}_{[t_i,t_{i+1})}(t)-\mathcal{L}_T\mathcal{A}I\left(\sum_{i = 1}^{k_n}\tilde{F}_i^n\mathbbm{1}_{[t_i,t_{i+1})}\right)(X_{\wedge T}).
\end{align*}

As explained next, the convergence of each of these terms is then verified. First, $DF$ is boundedness-preserving, and from the partition taken, $X^n(t) \to X(t)$, almost everywhere, thus by the Dominated Convergence Theorem, the first two integrals converge to $\int_0^T DF(t,X_{\wedge t})\,\mathrm{d}t+\int_0^T \langle \nabla F(t,X(t)),\mu\rangle\,\mathrm{d}s$.

Next, observe that the functional and its space derivatives are left-continuous in time, and that $X(t^-)-X(t^n_i) \to 0$, uniformly for all $t \in (t^n_i,t^n_{i+1})$, once again, from the choice of the partition. So, if $G$ is left-continuous, $\vert\vert G(t^n_i,(X^n_{\wedge t_i^-})^{X(t^-)-X(t_i^-)}) - G(t,X_{\wedge t^-})\vert\vert_\infty\to 0$. This uniform convergence allows to replicate the argument in \cite{EISENBAUM} to obtain the convergence of the second integral to $\int_0^T \nabla F(t,X_{\wedge t^-})^T\Sigma^{1/2}\cdot\mathrm{d}B(t)$, using the Burkholder–Davis–Gundy’s inequality.

\noindent Since $||\Delta X(\cdot)||_2 > 1$ for finitely many $t$, the fourth integral converges to 
\begin{align*}
    \sum_{t \in [0,T]} (F(t,X_{\wedge t})-F(t,X_{\wedge t^-}))\mathbbm{1}_{||\Delta X(t)||_2 > 1}(t).
\end{align*}

\noindent Let us next analyze the integral with respect to the discontinuous martingale:

\begin{dmath*}
 \int_0^T \int_{||y||_2 \leq 1} \sum_{i =1}^{k_n} (F_i^n(X^n(t^-)+y)-F_i^n(X^n(t^-)))\mathbbm{1}_{[t_i,t_{i+1})}(t)\,\mathrm{d}\tilde{N}(t,y)\\
 = \int_0^T \int_{||y||_2 \leq 1} \sum_{i =1}^{k_n} \int_0^1 \langle\nabla F^n_i(X^n(t^-)+hy),y\rangle\,\mathrm{d}h \mathbbm{1}_{[t_i,t_{i+1})}(t)\,\mathrm{d}\tilde{N}(t,y)\\
 = \int_0^T \int_{||y||_2 \leq 1} \int_0^1 \sum_{i =1}^{k_n}\langle\nabla F^n_i(X^n(t^-)+hy),y\rangle\mathbbm{1}_{[t_i,t_{i+1})}(t)\,\mathrm{d}h \,\mathrm{d}\tilde{N}(t,y)
\end{dmath*}.

Once again, if $t \not \in \tau_n$ for any $n$, $\sum_{i =1}^{k_n}\langle\nabla F^n_i(X^n(t^-)+hy),y\rangle\mathbbm{1}_{[t_i,t_{i+1})}(t)$ converges uniformly to $\langle\nabla F(t,X_{\wedge t^-}^{hy}),y\rangle$, and the convergence rate is uniform for $t \not \in \tau_n$. Thus, this convergence occurs almost everywhere in $[0,T]$. Suppose first that $X$ is a.s. bounded by a constant $C > 1$, then $||X(t^-)+hy||_2 \leq 2C$, and:

\begin{dmath*}
\mathbb{E}\left(\int_0^T \int_{||y||_2 \leq 1}\int_0^1 \sum_{i = 1}^{k_n} \langle\nabla F^n_i(X^n(t^-)+hy),y\rangle\mathbbm{1}_{[t_i,t_{i+1})}(t)-\nabla \langle F(t,X_{\wedge t^-}^{hy}),y\rangle\,\mathrm{d}h\tilde{N}(dt,dy)\mathrm{dt}\right)^2\\
= \mathbb{E}\int_0^T \int_{||y||_2 \leq 1}\left(\int_0^1 \sum_{i = 1}^{k_n} \langle\nabla F^n_i(X^n(t^-)+hy),y\rangle\mathbbm{1}_{[t_i,t_{i+1})}(t)-\nabla \langle F(t,X_{\wedge t^-}^{hy}),y\rangle\,\mathrm{d}h\right)^2\nu(dy)\mathrm{dt}\\
\leq \sup_{t \in [0,T], \vert\vert x\vert\vert_\infty < C}\vert\vert \nabla F(t,x_{\wedge t})\vert\vert_2^2\,\,\int_0^T \int_{||y||_2 \leq 1} 4C^2y^2\nu(dy)dt < +\infty
\end{dmath*}.

\noindent Thus, by stochastic dominated convergence, this integral converges in probability to \[\int_0^T\int_{\{||y||_2 \leq 1\}} \left(F(t,X_{\wedge t^-}^y)-F(t,X_{\wedge t^-})\right)\tilde{N}(dt,dy).\]
\noindent For the first term from the $\mathcal{L}$ operator,

\begin{align*}
 \sum_{i = 1}^{k_n}\int_0^T\int_{\mathbb{R}}\frac{\partial^2}{\partial x_j^2} I_j\tilde{F}_i^n(B_s|_{B^j_s = x})\mathbbm{1}_{[t_i,t_{i+1})}(s)\,\mathrm{d}L_t^x(B^j) &= -\int_0^T\int_{\mathbb{R}}\sum_{i =1}^{k_n} \frac{\partial \tilde{F}_i^n}{\partial x^j}(B_s|_{B^j_s = x})\mathbbm{1}_{[t_i,t_{i+1})}(s)\,\mathrm{d}L_s^x(B^j),   
\end{align*}

\noindent then:

\begin{align}
    \frac{1}{2}\mathbb{E}\Big\vert\int_0^T&\int_{\mathbb{R}}\frac{d\tilde{F}_i^n(B_s|_{B^j_s = x})}{dx_j}\mathbbm{1}_{\{[t_i,t_{i+1})\}}(s)-\frac{\partial \tilde{G}_{t,B}}{\partial x_j}(B_s|_{B^j_s = x})\,\mathrm{d}L_s^x\Big\vert\nonumber \leq \left\vert \left\vert \frac{\partial \tilde{F}_i^n(B_s)}{dx_j}\mathbbm{1}_{\{[t_i,t_{i+1})\}}(t)- \frac{\partial \tilde{G}_{t,B}}{\partial x_j}(B_s)\right\vert\right\vert_{L,j}.
\end{align}

\noindent Moreover,

\begin{align*}
    \vert\vert f\vert\vert_{L,j} &\leq \left(2\sqrt{T}+\int_0^T \frac{\vert B^j_s\vert}{s}\,\mathrm{d}s\right)\vert\vert f\vert\vert_\infty.
\end{align*}

\noindent Therefore, since in the infinity norm \[\sum_{i = 1}^{k_n} \frac{\partial \tilde{F}_{i,B}^n(B_s)}{dx_j}\mathbbm{1}_{\{[t_i,t_{i+1})\}}(t) \to \frac{\partial \tilde{G}}{\partial x_j}(B_s),\]
the right-hand side of $(4.11)$ converges to 0, obtaining convergence to the desired integral in $(4.10)$. For the second integral in the $\mathcal L$ operator, define: 

\begin{align*}
    H_{n,j}(t,x) = &\int_{\vert\vert y\vert\vert \leq 1} \sum_{i = 1}^{k_n}\left[\left(I_j \tilde{F}_i^n(x+Ry|_{(Ry)_j = 0})\right)(x+(Ry)_je_j)\mathrm{d}z-I_j\tilde{F}_i^n(x)-(Ry)_j\frac{\partial}{\partial x_j}I_j\tilde{F}_i^n(x)\right]\mathbbm{1}_{[t_i,t_{i+1})}(t)\nu(\mathrm{d}y)\\
    = &\int_{\vert\vert y\vert\vert \leq 1} \sum_{i = 1}^{k_n} \left(\int_0^{Ry_j} \tilde{F}_i^n(x+ze_j)\,\mathrm{d}z-(Ry)_j\tilde{F}_i^n(x)\right)\mathbbm{1}_{[t_i,t_{i+1})}(t)\nu(\mathrm{d}y)\\
    = &\int_{\vert\vert y\vert\vert \leq 1} \sum_{i = 1}^{k_n} \left(\int_0^{Ry_j}\int_0^z \frac{\partial}{\partial x_j}\tilde{F}_i^n(x+he_j)\,\mathrm{d}h\mathrm{d}z\right)\mathbbm{1}_{[t_i,t_{i+1})}(t)\nu(\mathrm{d}y).
\end{align*}

\noindent Similarly, with $F$ instead of $F_n$ define $H_j$, and obtain:

\begin{dmath*}
    \vert H_{n,j}(t,x)-H_j(t,x)\vert\leq \int_{\vert\vert y\vert\vert \leq 1}  \left(\int_0^{Ry_j}\int_0^z \left\vert\sum_{i = 1}^{k_n}\frac{d}{dx_j}F_i^n(x+he_j)\mathbbm{1}_{[t_i,t_{i+1})}(t)
    -\partial_j F(t,X_{\wedge t^-}^{X(t)+h-X(t^-)})\right\vert\,\mathrm{d}h\mathrm{d}z\right)\nu(\mathrm{d}y)
\end{dmath*}.

Once again, assume that $X$ and $B$ have bounded paths (to recover the general case, argue by stopping times). Thus, the difference in the previous equation is bounded by $C$, and $\vert\vert H_n(t,x)-H(t,x)\vert\vert_2 \leq C\vert\vert R\vert\vert_o\int_{\vert\vert y\vert\vert \leq 1} \vert\vert y\vert\vert_2^2\nu(dy)$ almost everywhere in $[0,T]$, where $\vert\vert \cdot\vert\vert_o$ denotes the operator norm, $\vert\vert R\vert\vert_o := \sup\{\vert\vert Rx\vert\vert_2: \vert\vert x\vert\vert_2 = 1\}$. Therefore, since $\vert\vert f\vert\vert_{L,j} \leq C\vert\vert f\vert\vert_\infty$, the convergence of this integral is obtained as in the case of the previous one.\\

\noindent For the last term,

\begin{align*}
    \sum_{i = 1}^{k_n} &\int_{\vert\vert y\vert\vert_2 \leq 1}\left(F^n_i(X(t^-)+(I-Q)y)-F^n_i(X(t^-))-\langle \nabla F_i^n(X^n(t^-)),(I-Q)y\rangle\right)\,\nu(\mathrm{d}y)\mathrm{d}s\mathbbm{1}_{[t_i,t_{i+1})}(t)\\
    = &\sum_{i = 1}^{k_n} \int_{\vert\vert y\vert\vert_2 \leq 1}\int_0^1 \left(\langle\nabla F^n_i(X(t^-)+s(I-Q)y),(I-Q)y\rangle-\langle \nabla F_i^n(X^n(t^-)),(I-Q)y\rangle\right)\mathrm{d}s\mathbbm{1}_{[t_i,t_{i+1})}(t)\,\nu(\mathrm{d}y)\mathrm{d}s.
\end{align*}

\noindent Since $X^n_t$ can be assumed to be bounded, $\nabla F$ is left continuous in time, and $X^n(t) \to X(t)$ a.e. in $[0,T]$, then by the dominated convergence theorem this last integral converges to:

\begin{align*}
    \int_0^T\int_{\vert\vert y\vert\vert_2 \leq 1}&\int_0^1 \left(\langle\nabla F(t,X_{t^-}^{s(I-Q)y}),(I-Q)y\rangle-\langle \nabla F(X^n_{t^-}),(I-Q)y\rangle\right)\mathrm{d}s\mathbbm{1}_{[t_i,t_{i+1})}(t)\,\nu(\mathrm{d}y)\mathrm{d}s.
\end{align*}

Which is indeed the last convergence needed to obtain the terms in $(4.9)$, giving the equation for the case when $X$, and $B$ are bounded. The general case is obtained since if $T_M := \inf\{t > 0: \vert\vert X(t)\vert\vert_2 \vee \vert\vert B(t)\vert\vert_2 > M\}$, the theorem holds locally for $X(t\wedge T_M)$, and thus for $X(t)$, by taking $M \to \infty$.
\end{proof}

The decomposition of functionals of weak Dirichlet processes presented in \cite{bouchardOctober} shows that if $F \in C^{0,1}$, then $F(t,X_{\wedge t})$ can be written as the sum of a local martingale and of an orthogonal process. Although the hypotheses to obtain such
a decomposition are more general than the existence of the horizontal derivative, using Theorem 4.4, the existence of the horizontal derivative allows for a characterization of the orthogonal component.

Combining Theorem $3.3$, and Theorem $4.4$, leads to:

\begin{Theorem}
Let $F \in C^{0,1}$ be differentiable in the direction of the g-Lipschitz $\gamma \in C^{0,0}$, and let $(X(t))_{t \in [0,T]}$ be a L\'{e}vy process with triplet $(\mu,\Sigma,\nu)$. Let $Q$, the projection operator onto the range of $\Sigma^{1/2}$, be such that $\int_{\vert\vert y\vert\vert\leq 1} \vert\vert(I-Q)y\vert\vert_2\,\nu(\mathrm{d}y)<\infty$, and let $G(t,x):  = F(t,X_{\wedge t}^{x-x(t^-)})$. Then,

\begin{align}
F(t&,X_{\wedge t})-F(0,X_{\wedge 0}) = \int_0^t D^\gamma F(s,X_{\wedge s})\,\mathrm{d}s+\int_0^t \langle \nabla F(s,X_{\wedge s}),\mu-\gamma(s,X_{\wedge s})\rangle\,\mathrm{d}s\nonumber\\
&+\int_0^t \nabla F(s,X_{\wedge s^-})^T\Sigma^{1/2}\,\mathrm{d}B(s)
+ \sum_{s \leq t} (F(s,X_{\wedge s})-F(s,X_{\wedge s^-}))\mathbbm{1}_{\{\vert\vert\Delta X(s)\vert\vert_2 > 1\}}\nonumber\\
&+\int_0^t \int_{||y||_2 \leq 1} \left(F(s,X_{\wedge s^-}^y)-F(s,X_{\wedge s^-})\right)\,\mathrm{d}\tilde{N}(ds,dy)-\mathcal{L}_t\mathcal{A}I\tilde{G}(B_{\wedge t})\nonumber\\
&+\int_0^t\int_{\vert\vert y\vert\vert_2 \leq 1} \left(F(s,X_{\wedge s^-}^y)-F(s,X_{\wedge s^-}^{Qy})-\langle \nabla F(t,X_{\wedge t^-}),(I-Q)y\rangle\right)\,\nu(\mathrm{d}y)\,\mathrm{d}s.
\end{align}

\end{Theorem}

\begin{remark}
\hfill\begin{enumerate}[(i)]
\item If $\Sigma$ is a $d\times d$ invertible matrix, then $Q = I$, and the condition $\int_{\vert\vert y\vert\vert_2 \leq 1} \vert\vert (I-Q)y\vert\vert_2 \nu(\mathrm{dy}) < \infty$ is immediately satisfied.

\item Given $F \in C^{0,2}$, such that $F$, and $\partial_i F, i = 1,...,d$ are all differentiable in the direction of $\gamma \in C^{0,0}$, then if $(X(t))_{t \in [0,T]}$ is a.s. a c\`{a}dl\`{a}g semimartingale, the corresponding Fisk-Stratonovich formula is given by:
\begin{align}
    F(t,X_{\wedge t}) - F(0,X_{\wedge 0}) =  &\int_0^t D^\gamma F(s,X_{\wedge s})\,\mathrm{d}s-\int_0^t \langle \nabla F(s,X_{\wedge s}),\gamma(s,X_{\wedge s})\rangle\,\mathrm{d}s+\int_0^t \nabla F((s,X_{\wedge s^-})\circ\,\mathrm{d}X(s) \nonumber\\
    &+\sum_{s \in [0,T]} \left(F(s,X_{\wedge s})-F(s,X_{\wedge s^-})-\langle \Delta X(s^-),\nabla F(s,X_{\wedge s^-})\rangle\right), \,\,\text{$\mathbf{P}$-a.s.}
\end{align}

\item The multivariate equivalent to the function space used in \cite{Dupire09} can be defined as $\Lambda := \bigcup_{s \in [0,T]}\Lambda_s$, where $\Lambda_s = D([0,s],\mathbb{R}^d)$, while our framework deals with functions defined in $[0,T]\times D([0,T],\mathbb{R}^d)$. Functions from one space to the other can be shown to be equivalent under the identifications $u: \Lambda \to [0,T]\times D([0,T],\mathbb{R}^d)$ and $v: [0,T]\times D([0,T],\mathbb{R}^d) \to \Lambda$ such that,

\begin{align*}
    u(w_t) &= (t,w(\cdot)\mathbbm{1}_{[0,t)}(\cdot)+w(t)\mathbbm{1}_{[t,T]}(\cdot)),
\end{align*}

i.e., the function $w_t$ with domain $[0,t]$ is mapped to the pair on the right-hand side, which has a path component given by $w|_{[0,t)]}1_{[0,t)}+w(t)1_{[t,T]}$. Moreover, the following identification can also be defined,
\begin{align*}
    v(t,w_{\wedge t}) &= w|_{[0,t]}.
\end{align*}
Given any $w_t \in \Lambda_t$, note that the differential equation,
\begin{align*}
    dy &= \gamma(s,y_{\wedge s})ds \,\,\text{for $s \in (t,T]$},\\
    y(s) &= w(s) \,\,\text{for $s \in [0,t]$},
\end{align*}
\noindent has an unique solution for $\gamma$ g-Lipschitz. Thus $y_s := y|_{[0,s]} \in \Lambda_s$ can be defined for all $s \in (t,T]$, and the derivative in the direction of $\gamma$ is given by,
\begin{align*}
    D^\gamma F_t (y_t) = \frac{F_{t+h}(y_{t+h})-F_t(y_t)}{h}.
\end{align*}

\item Let us next explain how the L\'{e}vy case, without Brownian component, recovers the Brownian one. First, in the definition of $\vert\vert\cdot\vert\vert_*$, let the metric $d_D$ be the one induced by the norm of the complete Skorokhod space. For $F \in C^{1,2}$, $F(t,\cdot)$ is continuous with respect to this norm, and thus given a sequence of functions $(X^n)_{n \geq 1}$ such that $X^n \xrightarrow[n \to \infty]{\mathcal{L}}X$, where $\mathcal{L}$ indicates convergence in law as elements of the Skorokhod space, it follows that

\begin{align*}
    \lim_{n \to \infty}\mathbb{E}[F(t,X^n_{\wedge t})] = \mathbb{E}[F(t,X)].
\end{align*}

Next, following the construction in \cite[Theorem 2.5]{GaussianApprox}, take a measurable family $\{\mu(\cdot|u): u\in S^{d-1}\}$ of L\'{e}vy measures on $(0,+\infty)$, and a finite positive measure $\lambda$ in the unit sphere $S^{d-1}$ whose support is not contained in any hyperplane, such that together they satisfy the condition,
\begin{align*}
    \lim_{\epsilon \to 0^+} \frac{1}{\epsilon^2}\int_0^\epsilon r^2\mu(\mathrm{d}r|u) = \infty,\, \lambda-\text{a.e}.
\end{align*}

\noindent Then, the L\'{e}vy measure $\tilde{\nu}_\epsilon$ defined via:

\begin{align*}
    \tilde{\nu}_\epsilon(\mathrm{d}r,\mathrm{d}u) &:= \mathbbm{1}_{\{r<\epsilon\}}\mu(\mathrm{d}r|u)\lambda(du),\, r>0, u \in S^{d-1},
\end{align*}

\noindent fulfills the conditions of \cite[Theorem 2.2]{GaussianApprox}. Then, if  $\tilde{b}_\epsilon = -\int_{\vert\vert \Sigma_\epsilon^{1/2}x\vert\vert\geq 1}\Sigma_\epsilon^{-1/2}x\tilde{\nu}_\epsilon(\mathrm{d}x)$, the L\'{e}vy processes $(\tilde{X}^\epsilon(t))_{t \in [0,T]}$ with characteristic triplet $(\tilde{b}_\epsilon,0,\tilde{\nu}_\epsilon)$ are such that $X^\epsilon := \Sigma_\epsilon^{-1/2} \tilde{X}^\epsilon \xrightarrow[\epsilon \to 0^+]{\mathcal{L}} B$, where $B$ is a multivariate standard Brownian motion, and moreover the L\'{e}vy process $(X^\epsilon(t))_{t \in [0,T]}$ has triplet $(b_\epsilon,0,\nu_\epsilon)$. The fact that $ X^\epsilon \xrightarrow[\epsilon \to 0^+]{\mathcal{L}} B$, is now used to show that for any $F \in C^{1,2}$, $F(T,B_{\wedge T})$ has the same distribution as the one given by the functional It\^{o} formula. Indeed, without loss of generality assume that $F$ and all its derivatives are bounded and, moreover, since then $\lim_{\epsilon \to 0^+} \int_0^T \mathbb{E}[DF(t,X^\epsilon_{\wedge t})]\,\mathrm{d}t = \int_0^T \mathbb{E}[DF(t,B_{\wedge t})]\,\mathrm{d}t$, assume further that $DF = 0$, then,

\begin{align}
    \mathbb{E}\left[F(T,X^\epsilon_{\wedge T})\right]-\mathbb{E}[F(0,X^\epsilon_{\wedge 0})] &= \int_0^T\int_{\mathbb{R}^d\setminus\{0\}} \mathbb{E}\left[F(t,X^{\epsilon, u}_{\wedge t})-F(t,X^\epsilon_{\wedge t^-})-\langle \nabla F(t,X^\epsilon_{\wedge t}), \mu\rangle\right]\,\nu_\epsilon(\mathrm{d}u)\mathrm{d}t\nonumber\\
    &= \int_0^T \int_{\mathbb{R}^d\setminus\{0\}} \mathbb{E}\left[\int_0^1 Tr(\nabla^2 F(t,X^{\epsilon, su}_{\wedge t^-})uu^t)(1-s)\,\mathrm{d}s\right]\nu_\epsilon(\mathrm{d}u)\mathrm{d}t.
\end{align}

By Lemma 2.2, the second derivatives $\partial_i\partial_j F$ are uniformly continuous therefore, take $\epsilon' > 0$, and $\kappa$ such that for any $1\leq i,j \leq d$, $\vert\vert (t,X_{\wedge t})-(s,Y_{\wedge s})\vert\vert_* < \kappa$, implies

\begin{align*}
    \vert \partial_i\partial_jF(t,X_{\wedge t^-})-\partial_i\partial_jF(s,Y_{\wedge s^-})\vert < \epsilon'.
\end{align*}

Next, note that $\int_{\mathbb{R}^d\setminus\{0\}} u u^t \,\nu_\epsilon(\mathrm{d}u) = Id_d$, and without loss of generality assume the derivatives $\partial_i\partial_j F$ to be bounded by a constant $M > 0$. Then, proceeding as in the proof of Remark 2.2(i) in \cite{Stein},

\begin{align*}
    \Big\vert\int_0^T \int_{\mathbb{R}^d\setminus\{0\}} \int_0^1 Tr(\nabla^2 F(t,&X^{su}_{\wedge t^-})uu^t)(1-s)\,\mathrm{d}s\nu_\epsilon(\mathrm{d}u)\mathrm{d}t-\frac{1}{2}\int_0^T \int_{\mathbb{R}^d\setminus\{0\}} Tr(\nabla^2 F(t,X_{\wedge t^-})uu^t)\,\nu_\epsilon(\mathrm{d}u)\mathrm{d}t\Big\vert\\
    &\leq \int_0^T \int_{0 < \vert\vert u\vert\vert_2 \leq \kappa} \int_0^1 \epsilon' \vert\vert u\vert\vert_2^2(1-s)\,\mathrm{d}s\nu_\epsilon(\mathrm{d}u)\mathrm{d}t+2dM\int_0^T  \int_{\vert\vert u\vert\vert_2 > \kappa}\vert\vert u\vert\vert_2^2\,\nu_\epsilon(\mathrm{d}u)\mathrm{d}t\\
    &= \frac{T\epsilon'}{2}\int_{0<\vert\vert u \vert\vert_2 \leq \kappa} \vert\vert u\vert\vert_2^2\nu_\epsilon(\mathrm{d}u)+2dMT\int_{\vert\vert u\vert\vert_2 > \kappa} \vert\vert u\vert\vert^2_2\nu_\epsilon(\mathrm{d}u) \longrightarrow 0,
\end{align*}
\noindent after first taking $\epsilon \to 0^+$, and then $\epsilon' \to 0^+$. Then, since $X^\epsilon \xrightarrow[\epsilon \to 0^+]{\mathcal{L}} B$,
\begin{align}
    \lim_{\epsilon \to 0^+} \mathbb{E}[F(T,X^\epsilon_{\wedge T})] &= \mathbb{E}[F(0,0)]+\int_0^T \mathbb{E}[Tr(\nabla^2F(t,B_{\wedge t}))]\,\mathrm{d}t.
\end{align}
\noindent Therefore, if $f:\mathbb{R}\to\mathbb{R}$ is a bounded, infinitely differentiable function, with all its derivatives also bounded, then $f\circ F \in C^{1,2}$, and $(4.17)$ gives,

\begin{align}
    \mathbb{E}[(f\circ F)(T,B_{\wedge T})] &= \mathbb{E}[(f\circ F)(0,B_{\wedge 0})]+\int_0^T \mathbb{E}\left[(f'\circ F)(t,B_{\wedge t})DF(t,B_{\wedge t})\right]\,\mathrm{d}t\nonumber\\
    +&\frac{1}{2}\int_0^T \mathbb{E}\left[Tr\left((f''\circ F)(t,B_{\wedge t})\nabla F(t,X_{\wedge t})\nabla^t F(t,B_{\wedge t})+f'(t)\nabla^2F(t,B_{\wedge t})\right)\right]\,\mathrm{d}t.
\end{align}

Let $(Z(t))_{t \in [0,T]}$ be defined by,

\begin{align*}
    Z(t) &:= F(0,B_{\wedge 0})+\int_0^t \left(DF(s,_{\wedge s})+\frac{1}{2}Tr(\nabla^2 F(s,B_{\wedge s}))\right)\,\mathrm{d}s+\int_0^t \nabla F(s,B_{\wedge s})\cdot \mathrm{d}B(s).
\end{align*}

\noindent Then, a direct application of the classical It\^{o}'s formula shows that for any $f$ as above, $\mathbb{E}[f(Z)]$ is equal to the right-hand side of $(4.18)$ and therefore, $Z(T)$ has the same distribution as $F(T,B_{\wedge T})$. Finally, by applying the Skorokhod representation theorem, there exist random variables $Y^\epsilon$ with the same distribution as $F(T,X^\epsilon_{\wedge T})$ converging a.s.\@ to a random variable having the same distribution as $Z$.
\end{enumerate}
\end{remark}
\section{Some Applications}

This next section presents two simple applications of the functional It\^{o} formulas obtained above, one in the case of L\'{e}vy processes, the other involving the derivative in the direction of $\gamma$. (Many more will be presented elsewhere.) To start, let us consider an extension to L\'{e}vy processes of the better pricing PDE for Asian options found in \cite{Dupire09}. Below, $R$ is defined prior to Definition 4.2, while $Q$ is the orthogonal projection as above.\\

\begin{proposition}
Let $(X(t))_{t \in [0,T]}$ be a multivariate L\'{e}vy process, which under $\mathbf{P}$ has triplet $(\mu,\Sigma,\nu)$, such that $\int_{\mathbb{R}^d\setminus\{0\}} \vert\vert x\vert\vert_2^2 \nu(\mathrm{d}x) < \infty$, $\int_{\vert\vert y\vert\vert\leq 1} \vert\vert(I-Q)y\vert\vert_2\,\nu(\mathrm{d}y)<\infty$, and $X(t) = \Sigma B(t)+\int_0^t \int_{\mathbb{R}^d\setminus \{0\}} y \tilde{N}(\mathrm{d}s,\mathrm{d}y)$. Let $J(t) = (J^1(t),...,J^m(t))$, $J^i(t):= \mathbb{E}[\int_0^T X^i(s)\,\mathrm{d}s|\mathcal{F}_t] = \int_0^t X^i(s)\,\mathrm{d}s+(T-t)X^i(t)$, and let the continuously differentiable $f: \mathbb{R}^+\times \mathbb{R}^{2d} \to \mathbb{R}$ be such that $F(t,X_{\wedge t}) = f(t,J(t), x(t))$ is the pricing option of an Asian option. Then,

\begin{align*}
    \frac{\partial f}{\partial t} =& \sum_{i = 1}^m \int_{\mathbb{R}} \left((T-t)\frac{\partial \tilde{f}}{\partial J_i}+\frac{\partial \tilde{f}}{\partial x_i}\right)\Big|_{B^i(t) = x}\mathrm{d}L^x_s(B^i)\\
    &+\int_{\mathbb{R}}\int_{\vert\vert y\vert\vert \leq 1}\int_0^1 \Bigg(\left((T-t)\frac{\partial f}{\partial J_i}+\frac{\partial f}{\partial x_i}\right)(t,J(t)|_{B^i(t)= x}+(T-t)sRy,X(t)|_{B^i(t) = x}+sRy)\\
    &-\left((T-t)\frac{\partial f}{\partial J_i}+\frac{\partial f}{\partial x_i}\right)(t,J(t)|_{B^i(t) = x},X(t)|_{B^i(t) = x})\Bigg)(Ry)_i\nu(\mathrm{d}y)\,\mathrm{d}L^x_s(B^i)\\
    -&\int_{\vert\vert y \vert\vert_2 > 1} \left(f(t,,J(t^-)+(T-t)y,X(t^-)+y)-f(t,J(t^-),X(t^-))-\langle (T-t)\nabla_J f+\nabla_x f,y\rangle\right) \nu(\mathrm{d}y)\\
    -&\int_{\vert\vert y \vert\vert_2 \leq 1} \Big(f(t,,J(t^-)+(T-t)y,X(t^-)+y)-f(t,J(t^-)+(T-t)Qy,X(t^-)+Qy)\\
    &-\langle (T-t)\nabla_J f+\nabla_x f,(I-Q)y\rangle\Big) \nu(\mathrm{d}y).
\end{align*}
\end{proposition}

\begin{proof}

For each $t$,
\begin{align*}
    X(t) = \Sigma B(t)-\int_0^t\int_{\vert\vert y\vert\vert_2 > 1} y\nu(\mathrm{d}t)+\int_0^t\int_{\vert\vert y\vert\vert_2 > 1}y N(\mathrm{d}s,\mathbb{d}y)+\int_0^t\int_{\vert\vert y\vert\vert_2 \leq 1}y \tilde{N}(\mathrm{d}s,\mathbb{d}y). 
\end{align*}
\noindent Therefore, applying Theorem 4.2.\ to the functional $F \in C^{1,1}$ given by $F(t,X_{\wedge t}) = f(t,J(t),X(t))$, the following identity is obtained in differential notation:
\begin{align}
    dF(t,X_{\wedge t}) &= \langle\nabla F(t,X_{\wedge t^-}),\Sigma dB(t)\rangle +\int_{\mathbb{R}^d\setminus\{0\}} \left(F(t,X_{\wedge t})-F(t,X_{\wedge t^-})\right)\tilde{N}(\mathrm{d}s,\mathrm{d}y)\nonumber\\
    &+DF(t,X_{\wedge t})\mathrm{d}t-\sum_{i = 1}^m \int_{\mathbb{R}}\mathcal{A}_iI\tilde{G}(B(t)|_{B^i(t) = x})\,\mathrm{d}L^x_s(B^i)\nonumber\\
    &+\int_{\vert\vert y \vert\vert_2 > 1} \left(F(t,X_{\wedge t^-}^y)-F(t,X_{\wedge t^-})-\langle \nabla F(t,X_{\wedge t^-}),y\rangle\right) \nu(\mathrm{d}y)\mathrm{d}t\nonumber\\
    &+\int_{\vert\vert y \vert\vert\leq 1} \left(F(t,X_{\wedge t^-}^y)-F(t,X_{\wedge t^-}^{Qy})-\langle \nabla F(t,X_{\wedge t^-}),(I-Q)y\rangle\right) \nu(\mathrm{d}y)\mathrm{d}t. 
\end{align}
\end{proof}

Since the pricing function $F(t,X_{\wedge t})$ ought to be a martingale under the given measure, then it would be enough to show that all but the first two terms are of bounded variation. Since $DF$ is locally bounded, the integral of the term $DF(t,X_{\wedge t})$ is of bounded variation, moreover, almost surely, the same is true for the terms involving the measure with respect to the local time, from the inequality (4.7) and the local bounds on the derivatives. The next to last integral can be split into $\int_{\vert\vert y\vert\vert_2 > 1} \left(F(s,X_{\wedge s^-}^y)-F(s,X_{\wedge s^-})\right)\nu(\mathrm{d}y)$, which is bounded from the finiteness of $\nu(\mathbb{R}^d\setminus \bar{B}_1)$, and $\int_{\vert\vert y\vert\vert_2 \geq 1} \langle \nabla F(t,X_{\wedge t^-}),(I-Q)y\rangle\nu(\mathrm{d}y)$ which is finite due to the finite second moment condition of the hypothesis. Moreover, using ideas similar to those after equation $(4.12)$, the last integral is also of bounded variation, as a consequence of the condition $\int_{\vert\vert y\vert\vert\leq 1} \vert\vert(I-Q)y\vert\vert_2\,\nu(\mathrm{d}y)<\infty$.\\

\noindent Therefore, since,

\begin{align}
    &DF(t,X_{\wedge t}) = \frac{\partial f}{\partial t}(t,J(t),X(t)),\\
    &\nabla F(t,X_{\wedge t})= \Big((T-t)\nabla_J f + \nabla_x f\Big)(t,J(t),X(t))](t,J(t),X(t)),\\
    &F(t,X_{\wedge t^-}^y)-F(t,X_{\wedge t^-}) = f(t,J(t^-)+(T-t)y,X(t^-)+y)-f(t,J(t^-),X(t^-)),
\end{align}

\noindent Theorem 5.1 is obtained, as the bounded variation component in $(5.1)$ is zero.\\

The result below showcases a problem that cannot be tackled by means of the horizontal derivative alone, but where the derivative in the direction of a functional allows for an application of the functional It\^{o} formula. First, a definition is in order.

\begin{definition}
A left-continuous, non-anticipating, boundedness preserving, $h$-Lipschitz functional $\gamma: \mathbb{R}^d \to \mathbb{R}^d$ is said to ignore a single jump, if for all $t \in [0,T]$, $x \in D([0,T],\mathbb{R}^d)$, $\gamma(t,x_{\wedge t}) = \gamma(t,(x(\cdot)-\Delta x(s)\mathbbm{1}_{\{s\}}(\cdot))_{\wedge t})$, for all $s \leq t$.
\end{definition}

The next theorem provides an integral form for any functional which is constant along the curves with derivative given by $\gamma$. Clearly, the case of ignoring finitely many jumps follows with the same approach.

\begin{proposition}
Let $f: \mathbb{R}^d \to \mathbb{R}$ be continuously differentiable, let $\gamma$ be single jump ignoring, and let $F:D([0,T],\mathbb{R}^d) \to \mathbb{R}$ such that $F(t,X_{\wedge t}) := f\left(X(t)+\int_t^T \gamma(s,Y^{s,X}_{\wedge s})\,\mathrm{d}s\right)$, where $Y^{s,X}$ is as in Definition 3.1 and where $(X(t))_{t \in [0,T]}$ is $\mathbf{P}$-a.s. a continuous semimartingale. Then, 

\begin{align}
    f(X(T))-f\left(X(0)+\int_0^T \gamma(s,Y_{\wedge s}^{s,X})\,\mathrm{d}s\right) = \int_0^T \nabla F(t,X_{\wedge t^-})\cdot\mathrm{d}X(t)-\int_0^T \langle \nabla F(t,X_{\wedge t}), \gamma(t,X_{\wedge t})\rangle\,\mathrm{d}t& \nonumber\\
    +\int_0^T Tr(\nabla^2 F(t,X_{\wedge t}) \mathrm{d}[X](t)),&
\end{align}

\noindent where, $\nabla F(t,X_{\wedge t}) = \nabla f\left(X(t)+\int_t^T \gamma(s,Y^{t,X})\mathrm{d}s\right)$, and $\partial_i\partial_j F(t,X_{\wedge t}) = \frac{\partial^2}{\partial x_i\partial x_j}f\left(X(t)+\int_t^T \gamma(s,Y^{t,X}_{\wedge s})\,\mathrm{d}s\right)$.
\end{proposition}

\begin{proof}
\begin{align*}
F(t+h,Y^{t, X}_{\wedge t+h}) &= f\left(Y^{t,X}(t+h)+\int_{t+h}^T \gamma(s,Y_{\wedge s}^{t+h,X})\,\mathrm{d}s\right)\\
&= f\left(X(t)+\int_{t}^T \gamma(s,Y_{\wedge s}^{t+h,X})\,\mathrm{d}s\right) = F(t,X_{\wedge t}).
\end{align*}

\noindent Thus, $D^\gamma F(t,X_{\wedge t}) = 0$, and the property of ignoring a single jump allows to obtain $\nabla F$ and $\partial_i\partial_j F$ as noted above. Then, $(5.5)$ follows from a direct application of the functional It\^{o} formula using the derivative in the direction of $\gamma$.
\end{proof}

\printbibliography

@article{Dupire09,
    author  = "Dupire, B.",
    title   = "Functional {I}t\^{o} Calculus",
    year    = "2009",
    journal = "SSRN",
    note={Reprinted in Quantitative Finance, 5, pages 721--729 (2019)} 
}

@phdthesis{Fournie10,
  author       = {Fourni\'{e}, D.}, 
  title        = {Functional {I}t\^{o} calculus and Applications},
  school       = {Columbia University},
  year         = 2010
}

@article{Follmer81,
     author = {F\"ollmer, H.},
     title = {Calcul d'{I}t\^{o} sans probabilit\'es},
     journal = {S\'eminaire de probabilit\'es de Strasbourg},
     pages = {143--150},
     publisher = {Springer - Lecture Notes in Mathematics},
     volume = {15},
     year = {1981},
     zbl = {0461.60074},
     mrnumber = {622559}
}

@article{Eisenbaum09,
author = {Eisenbaum, N. and Walsh, A.},
title = {{An optimal {I}t\^{o} formula for L\'{e}vy processes}},
volume = {14},
journal = {Electronic Communications in Probability},
publisher = {Institute of Mathematical Statistics and Bernoulli Society},
pages = {202 -- 209},
year = {2009},
}

@article{Bouchard,
      title={A ${C}^{0,1}$-functional {I}t\^{o}'s formula and its applications in mathematical finance}, 
      author={Bouchard, B. and Loeper, G. and Tan, X.},
      year={2021},
      eprint={https://arxiv.org/pdf/2101.03759.pdf},
      archivePrefix={arXiv},
      primaryClass={math.PR}
}

@article{bouchardOctober,
      title={{I}t{\^o}-{D}upire's formula for ${C}^{0,1}$-functionals of c{\`a}dl{\`a}g weak {D}irichlet processes}, 
      author={B. Bouchard and M. Vallet},
      year={2021},
      eprint={https://arxiv.org/pdf/2110.03406.pdf},
      archivePrefix = {arxiv},
      primaryClass={math.PR}
}

@article{Cont13,
author = {Cont, R. and Fourni\'{e}, D.},
title = {{Functional {I}t\^{o} calculus and stochastic integral representation of martingales}},
volume = {41},
journal = {The Annals of Probability},
number = {1},
publisher = {Institute of Mathematical Statistics},
pages = {109 -- 133},
year = {2013},
}

@article{LEVENTAL13,
title = {A simple proof of functional {I}t\^{o}’s lemma for semimartingales with an application},
journal = {Statistics \& Probability Letters},
volume = {83},
number = {9},
pages = {2019-2026},
year = {2013},
issn = {0167-7152},
author = {Levental, S. and Schroder, M. and Sinha, S.}
}

@article{ConvexRisk,
title = {A functional {I}t\^{o}’s calculus approach to convex risk measures with jump diffusion},
journal = {European Journal of Operational Research},
volume = {250},
number = {3},
pages = {874-883},
year = {2016},
issn = {0377-2217},
author = {Siu, T.}
}

@article{Oberhauser,
author = {Oberhauser, H.},
title = {The functional {I}t\^{o} formula under the family of continuous semimartingale measures},
journal = {Stochastics and Dynamics},
volume = {16},
number = {04},
pages = {1650010},
year = {2016}
}

@article{Zhang,
author = {Ekren, I. and Keller, C. and Touzi, N. and Zhang, J.},
title = {{On viscosity solutions of path dependent PDEs}},
volume = {42},
journal = {The Annals of Probability},
number = {1},
publisher = {Institute of Mathematical Statistics},
pages = {204 -- 236},
year = {2014},
}

@book{Kim,
author = {Kim, A.},
title = {Functional Differential Equations},
subtitle = {Application of i-smooth calculus},
year = {1999},
publisher = {Springer Netherlands}
}

@article{CossoRusso,
author = {Cosso, A. and Russo, F.},
title = {Functional {I}t\^{o} versus {B}anach space stochastic calculus and strict solutions of semilinear path-dependent equations},
journal = {Infinite Dimensional Analysis, Quantum Probability and Related Topics},
volume = {19},
number = {4},
pages = {1650024},
year = {2016}
}

@book{applebaum_2009, 
place={Cambridge}, 
edition={2}, 
series={Cambridge Studies in Advanced Mathematics}, title={L\'{e}vy Processes and Stochastic Calculus},
publisher={Cambridge University Press}, 
author={Applebaum, D.}, 
year={2009}, 
collection={Cambridge Studies in Advanced Mathematics}}

@article{EISENBAUM,
title = {Local time–space stochastic calculus for {L}\'{e}vy processes},
journal = {Stochastic Processes and their Applications},
volume = {116},
number = {5},
pages = {757-778},
year = {2006},
issn = {0304-4149},
author = {Eisenbaum, N.}
}

@article{Ahn,
author = {Ahn, H.},
title = {{Semimartingale integral representation}},
volume = {25},
journal = {The Annals of Probability},
number = {2},
publisher = {Institute of Mathematical Statistics},
pages = {997 -- 1010},
year = {1997},
}

@article{Comparison,
title = {A note on functional derivatives on continuous paths},
author = {Ji, S. and Yang, S.},
year = {2015},
journal = {Statistics \& Probability Letters},
volume = {106},
number = {C},
pages = {176-183}
}

@book{Protter, 
title = {Stochastic Integration and Differential Equations},
author = {Protter, P.},
editor = {Rozovskii, B. and Yor, M.},
year = {2005},
publisher = {Springer-Verlag Berlin Heidelberg
},
volume = {21}
}

@article{Saporito,
author = {Saporito, Y.},
title = {The functional {M}eyer–{T}anaka formula},
journal = {Stochastics and Dynamics},
volume = {18},
number = {4},
pages = {1850030},
year = {2018}
}

@article{MemoryJumps,
title = {Stochastic systems with memory and jumps},
journal = {Journal of Differential Equations},
volume = {266},
number = {9},
pages = {5772-5820},
year = {2019},
issn = {0022-0396},
doi = {https://doi.org/10.1016/j.jde.2018.10.052},
url = {https://www.sciencedirect.com/science/article/pii/S0022039618306466},
author = {Baños, D. and Cordoni, F and {Di Nunno}, G. and {Di Persio}, L. and R{ø}se, E}
}

@book{Barcelona,
author = {Bally, V. and Caramellino, L. and Cont, R.},
editor = {Utzet, Frederic and Vives Josep},
publisher = {Birkh{\"a}user Basel},
year = {2016},
month = {01},
pages = {},
title = {Stochastic integration by parts and Functional {I}t\^{o} calculus}
}

@article{OberhauserStratonovich,
author = {Litterer, C. and Oberhauser, H.},
year = {2014},
title = {On a {C}hen–{F}liess approximation for diffusion functionals},
journal = {Monatshefte für Mathematik},
pages = {577-593},
volume = {175},
number = {4},
}

@article{CONT20101043,
title = {Change of variable formulas for non-anticipative functionals on path space},
journal = {Journal of Functional Analysis},
volume = {259},
number = {4},
pages = {1043-1072},
year = {2010},
issn = {0022-1236},
author = {Cont, R. and Fournié, D.}
}

@article{Stein,
author = {Arras, B. and Houdr\'{e}, C.},
title = {{On Stein’s method for multivariate self-decomposable laws}},
volume = {24},
journal = {Electronic Journal of Probability},
publisher = {Institute of Mathematical Statistics and Bernoulli Society},
pages = {1 -- 63},
year = {2019},
doi = {10.1214/19-EJP378},
URL = {https://doi.org/10.1214/19-EJP378}
}

@article{GaussianApprox,
author = {Cohen, S. and Rosinski, J.},
title = {{Gaussian approximation of multivariate Lévy processes with applications to simulation of tempered stable processes}},
volume = {13},
journal = {Bernoulli},
number = {1},
publisher = {Bernoulli Society for Mathematical Statistics and Probability},
pages = {195 -- 210},
year = {2007},
doi = {10.3150/07-BEJ6011},
URL = {https://doi.org/10.3150/07-BEJ6011}
}

@article{ViscosityKeller,
title = {Viscosity solutions of path-dependent integro-differential equations},
journal = {Stochastic Processes and their Applications},
volume = {126},
number = {9},
pages = {2665-2718},
year = {2016},
issn = {0304-4149},
doi = {https://doi.org/10.1016/j.spa.2016.02.014},
url = {https://www.sciencedirect.com/science/article/pii/S0304414916000442},
author = {Keller, C.}
}

@article{PathwiseTaylor,
title = {Pathwise Taylor expansions for random fields on multiple dimensional paths},
journal = {Stochastic Processes and their Applications},
volume = {125},
number = {7},
pages = {2820-2855},
year = {2015},
issn = {0304-4149},
doi = {https://doi.org/10.1016/j.spa.2015.02.004},
url = {https://www.sciencedirect.com/science/article/pii/S030441491500040X},
author = {Buckdahn, R. and Ma, J. and Zhang, J.}
}

@article{RoughPath,
title = {Pathwise Itô calculus for rough paths and rough PDEs with path dependent coefficients},
journal = {Stochastic Processes and their Applications},
volume = {126},
number = {3},
pages = {735-766},
year = {2016},
issn = {0304-4149},
doi = {https://doi.org/10.1016/j.spa.2015.09.018},
url = {https://www.sciencedirect.com/science/article/pii/S030441491500246X},
author = {Keller, C. and Zhang, J.}
}

@article{ViscosityI,
author = {Ekren, I. and Touzi, N. and Zhang, J.},
title = {{Viscosity solutions of fully nonlinear parabolic path dependent PDEs: Part I}},
volume = {44},
journal = {The Annals of Probability},
number = {2},
publisher = {Institute of Mathematical Statistics},
pages = {1212 -- 1253},
year = {2016},
doi = {10.1214/14-AOP999},
URL = {https://doi.org/10.1214/14-AOP999}
}

@article{ViscosityII,
author = {Ekren, I. and Touzi, N. and Zhang, J.},
title = {{Viscosity solutions of fully nonlinear parabolic path dependent PDEs: Part II}},
volume = {44},
journal = {The Annals of Probability},
number = {4},
publisher = {Institute of Mathematical Statistics},
pages = {2507 -- 2553},
year = {2016},
doi = {10.1214/15-AOP1027},
URL = {https://doi.org/10.1214/15-AOP1027}
}
\end{document}